\documentclass[oneside,a4paper,12pt,centertags]{amsart}
\usepackage{vmargin}
\usepackage{amssymb, amsthm, amscd, latexsym, amsmath, verbatim}
\usepackage[all]{xy}
\usepackage[mathscr]{eucal}
\usepackage{mathrsfs}
\usepackage[dvips]{color}
\usepackage[bookmarks=true]{hyperref}

\let\oldmarginpar\marginpar
\renewcommand\marginpar[1]{\-\oldmarginpar[\raggedleft\footnotesize #1]
{\raggedright\footnotesize #1}}

\theoremstyle{plain}
\newtheorem{Thm}{Theorem}[section]

\theoremstyle{definition}
\newtheorem{Prop}[Thm]{Proposition}
\newtheorem{Lem}[Thm]{Lemma}
\newtheorem{Cor}[Thm]{Corollary}
\newtheorem{Def}[Thm]{Definition}

\theoremstyle{remark}

\newcommand{\R}{\mathbb{R}}
\newcommand{\C}{\mathbb{C}}
\newcommand{\N}{\mathbb{N}}

\newcommand{\re}{\mathbb{\text{Re}}}
\newcommand{\supp}{\mathop{\text{sppt}}}
\newcommand{\esss}{\mathop{\text{ess sup}\,}}
\newcommand{\ep}{\epsilon}
\newcommand{\al}{\alpha}
\newcommand{\Dom}{\textsf{D}}
\newcommand{\Ran}{\textsf{R}}
\newcommand{\Nul}{\textsf{N}}
\newcommand{\Gra}{\textsf{G}}
\newcommand{\sgn}{\mathop{\text{sgn}}}
\renewcommand{\d}{\mathrm{d}}

\newcommand\ca{{\bf 1}}

\renewcommand{\div}{\mathop{\mathrm{div}}}

\newcommand{\Div}{\mathop{\mathrm{Div}}}
\newcommand{\grad}{\mathop{\mathrm{grad}}}

\newcommand{\tr}{\mathop{\mathrm{trace}}}
\newcommand{\End}{\mathop{\text{End}}}

\newcommand{\conj}{\mathop{\text{conj}}}
\def\XXint#1#2#3{{\setbox0=\hbox{$#1{#2#3}{\int}$}
    \vcenter{\hbox{$#2#3$}}\kern-.5\wd0}}
\def\Xint#1{\mathchoice
    {\XXint\displaystyle\textstyle{#1}}
    {\XXint\textstyle\scriptstyle{#1}}
    {\XXint\scriptstyle\scriptscriptstyle{#1}}
    {\XXint\scriptscriptstyle\scriptscriptstyle{#1}}
    \!\int}
\def\dashint{\Xint-}

\numberwithin{equation}{section}

\begin{document}
\title[The Kato Square Root Problem on Submanifolds]{The Kato Square Root Problem on Submanifolds}
\author[Morris]{Andrew J. Morris}

\address{Andrew J. Morris\\ Department of Mathematics\\ University of Missouri\\ Columbia \\ MO 65211\\ USA}
\email{morrisaj@missouri.edu}


\date{25 March 2011}

\begin{abstract}
We solve the Kato square root problem for divergence form operators on complete Riemannian manifolds that are embedded in Euclidean space with a bounded second fundamental form. We do this by proving local quadratic estimates for perturbations of certain first-order differential operators that act on the trivial bundle over a complete Riemannian manifold with at most exponential volume growth and on which a local Poincar\'{e} inequality holds. This is based on the framework for Dirac type operators that was introduced by Axelsson, Keith and McIntosh.
\end{abstract}

\maketitle

\tableofcontents

\section{Introduction and Main Results}
Let us briefly recall the Kato square root problem on $\R^n$. Given a strictly accretive matrix-valued function $A$ on $\R^n$ with bounded measurable coefficients, the Kato square root problem is to determine the domain of the square root $\sqrt{\div (A \nabla)}$ of the divergence form operator $\div(A \nabla)$. The original questions posed by Kato can be found in \cite{Kato1961,Kato} and are discussed further in \cite{McIntosh1990}. The problem was solved completely in the case $n=1$ by Coifman, McIntosh and Meyer in~\cite{CMcM1982}, in the case $n=2$ by Hofmann and McIntosh in~\cite{HMc} and finally for all $n\in\N$ by Auscher, Hofmann, Lacey, McIntosh and Tchamitchian in~\cite{AHLMT}. The reader is referred to the references within those works for the full list of attributes that led to those results, since it is not possible to include them all here.

Prior to the solution of the Kato problem in all dimensions, Auscher, McIntosh and Nahmod~\cite{AMcN1997} reduced the one dimensional problem to proving quadratic estimates for a related first-order elliptic system. Subsequently, Axelsson, Keith and McIntosh~\cite{AKMc} developed a general framework for proving quadratic estimates for perturbations of Dirac type operators on $\R^n$. In this unifying approach, the solution of the Kato problem in all dimensions as well as many results in the Calder\'{o}n program, such as the boundedness of the Cauchy singular integral operator on Lipschitz curves, follow as immediate corollaries. Their results also have applications to compact Riemannian manifolds (see Section~7 in \cite{AKMc}) and it is these applications that we extend to certain noncompact manifolds in this paper.

To state our results, let us fix the following notation. Let $M$ denote a complete Riemannian manifold with geodesic distance~$\rho$ and Riemannian measure~$\mu$. We adopt the convention that such a manifold is smooth and connected. The manifold is not required to be compact. For any smooth real vector bundle $E$ over $M$, let $C^\infty(E)$ denote the space of smooth sections of $E$, and let $C^\infty_c(E)$ denote the subspace of sections in $C^\infty(E)$ that are compactly supported. Given a smooth bundle metric on~$E$, where $\langle\cdot,\cdot\rangle_{E_x}:E_x\times E_x \rightarrow \R$ denotes the metric on the fibre $E_x$ of $E$ at each $x$ in $M$, let $L^\infty(E)$ denote the Banach space of all measurable sections $u$ of $E$ that satisfy
$
\|u\|_{L^\infty(E)} := \esss_{x\in M} |u(x)|_{E_x} < \infty,
$
where $|\cdot|_{E_x}$ denotes the norm induced by the bundle metric on $E_x$. Let $L^2(E)$ denote the Hilbert space of all measurable sections $u$ of $E$ that satisfy
$
\|u\|_{L^2(E)}^2 :=\int_M |u(x)|_{E_x}^2\,\d\mu(x) < \infty
$
with the inner inner-product
$
\langle u,v\rangle_{L^2(E)} := \int_M \langle u(x),v(x)\rangle_{E_x}\,\d\mu(x)
$
for all $u,v\in L^2(E)$. 

We assume that any real vector bundle has been complexified. For instance, when $E$ is the trivial bundle $M\times\R$, its complexification is $M\times\C$, so the spaces above consist of $\C$-valued measurable functions on $M$. In fact, set ${C^\infty(M):=C^\infty(M\times\C)}$, ${C^\infty_c(M):=C^\infty_c(M\times\C)}$, ${L^\infty(M):= L^\infty(M\times\C)}$ and ${L^2(M):= L^2(M\times\C)}$. 

We consider the following vector bundles over $M$. The tangent bundle $TM$, the cotangent bundle $T^*M$, the endomorphism bundle $\End(TM)$ and the tensor bundle $T^{k,l}M$ for each $k,\,l\in\N_0$. For each $x$ in $M$, the fibres of these bundles are, respectively, the tangent space $T_xM$, the cotangent space $T^*_xM$, the space $\End(T_xM)$ of endomorphisms on $T_xM$, and the space $T^{k,l}_xM:=\bigotimes^k T_xM \otimes \bigotimes^l T_x^*M$ of tensors. The smooth bundle metrics on $T^*M$, $\End(TM)$ and $T^{k,l}M$ are defined to be those induced by the Riemannian metric on $TM$.

The Sobolev space $W^{1,2}(M)$ of functions is defined in Section~\ref{Section: DiracTypeOperators}. The gradient and divergence on $M$ are defined in Section~\ref{Section: KatoApp} as closed operators
\begin{align*}
\grad&:\Dom(\grad) \subseteq L^2(M)\rightarrow L^2(TM), \\
\div&: \Dom(\div) \subseteq L^2(TM)\rightarrow L^2(M)
\end{align*}
with domain $\Dom(\grad)=W^{1,2}(M)$ and $-\div$ being the formal adjoint of $\grad$.

Given a function $A_{00}$ in $L^\infty(M)$, a vector field $A_{10}$ in $L^\infty(TM)$, a differential form $A_{01}$ in $L^\infty(T^*M)$, and $A_{11}$ in $L^\infty(\End(TM))$, define
$A:L^2(M)\oplus L^2(TM)\rightarrow L^2(M)\oplus L^2(TM)
$
by
\begin{equation}\label{eq: A.Def}
(Au)_x = \begin{bmatrix}(A_{00})_x & (A_{01})_x \\(A_{10})_x & (A_{11})_x\end{bmatrix}
\begin{bmatrix} (u_0)_x  \\ (u_1)_x\end{bmatrix}
\end{equation}
for all $u=(u_0,u_1)\in L^2(M) \oplus L^2(TM)$, where $(\cdot)_x$ denotes the value of a function or section at $x$ in $M$. The components $A_{00}$ and $A_{10}$ act by multiplication, as in $(A_{10})_x\big((u_0)_x\big):=(A_{10})_x \times u_0(x)$. The notation for the components of $A$ is chosen to reflect that $T^{0,0}M:=\C$, $T^{1,0}M=TM$, $T^{0,1}M=T^*M$ and $T^{1,1}M \cong \End(TM)$. The bilinear form ${J_A: W^{1,2}(M)\times W^{1,2}(M)\rightarrow\C}$ is then defined by 
\[
J_A(u,v) = \langle A_{11}(\grad u)+A_{10}u, \grad v \rangle_{L^2(TM)}
+\langle A_{01}(\grad u) + A_{00}u, v\rangle_{L^2(M)}
\]
for all $u,\,v\in W^{1,2}(M)$.

Given $A$ as above and $a$ in $L^\infty(M)$, suppose that there exist constants $\kappa_1$, $\kappa_2>0$ such that the following accretivity conditions are satisfied:
\begin{align}\begin{split}\label{eq: accr.a.A}
\re\, \langle a\,u, u\rangle_{L^2(M)} &\geq \kappa_1 \|u\|_{L^2(M)}^2 \quad\;\:\: \text{for all}\quad u\in L^2(M);\\
\re\, J_A(u,u) &\geq \kappa_2 \|u\|_{W^{1,2}(M)}^2
\quad\text{for all}\quad u\in W^{1,2}(M).
\end{split}\end{align}
The divergence form operator $L_{A,a}: \Dom(L_{A,a}) \subseteq L^2(M) \rightarrow L^2(M)$ is then defined by
\begin{equation}\label{eq: LA}
L_{A,a}u =  a\{-\div[A_{11} (\grad u) + A_{10}u] + A_{01}(\grad u) + A_{00}u\}
\end{equation}
for all $u\in \Dom(L_{A,a}) := \{u\in W^{1,2}(M) :A_{11} (\grad u) + A_{10}u \in\Dom(\div)\}$. We solve the Kato square root problem for the operator $L_{A,a}$ as in the following theorem.

\begin{Thm}\label{Thm: MainSubManifoldThm}
Let $n\in\N$ and suppose that $M$ is a complete Riemannian manifold that is embedded in $\R^n$ with a bounded second fundamental form. If $a$ and $A$ satisfy the accretivity conditions in \eqref{eq: accr.a.A}, then the divergence form operator $L_{A,a}$ defined by \eqref{eq: LA} has a square root $\sqrt{L_{A,a}}$ with domain $\Dom(\sqrt{L_{A,a}}) = W^{1,2}(M)$ and
\[
\|\sqrt{L_{A,a}} u\|_{L^2(M)} \eqsim \|u\|_{W^{1,2}(M)}
\]
for all $u\in W^{1,2}(M)$.
\end{Thm}

To prove Theorem~\ref{Thm: MainSubManifoldThm}, we develop a general framework for a class of first-order differential operators that act on the trivial bundle over a complete Riemannian manifold. This is the content of Section~\ref{Section: DiracTypeOperators}. The main result, Theorem~\ref{Thm: mainquadraticestimate}, is a local quadratic estimate for certain $L^\infty$ perturbations of these operators on manifolds with at most exponential volume growth and on which a local Poincar\'{e} inequality holds. This framework is based on that introduced in \cite{AKMc}, although it resembles more closely the subsequent development by the same authors in \cite{AKMc2}. The statement of Theorem~\ref{Thm: mainquadraticestimate} requires some technical preliminaries so we omit it here. 

The structure of the remainder of the paper is as follows. We obtain the solution of the Kato square root problem stated in Theorem~\ref{Thm: MainSubManifoldThm} as a corollary of Theorem~\ref{Thm: mainquadraticestimate} in Section~\ref{Section: KatoApp}. The technical tools required to prove Theorem~\ref{Thm: mainquadraticestimate} include a local version of the dyadic cube structure developed by Christ in \cite{Christ} and the local properties of Carleson measures. The relevant details are contained in Section~\ref{ChristCarleson} and Theorem~\ref{Thm: mainquadraticestimate} is proved in Section~\ref{Section: MainLQE}. The material in Sections~\ref{Section: DiracTypeOperators} and \ref{Section: MainLQE} follows closely the treatments in \cite{AKMc,AKMc2} and the reader is advised to have a copy of those papers at hand.

The following notation is used throughout the paper. For all $x,y\in\R$, we write $x\lesssim y$ to mean that there exists a constant $c\geq1$, which may only depend on constants specified in the relevant preceding hypotheses, such that $x\leq cy$. We also write $x\eqsim y$ to mean that $x\lesssim y \lesssim x$.

\section{Dirac Type Operators}\label{Section: DiracTypeOperators}
We begin by fixing some notation from operator theory. An operator $T$ on a Hilbert space $\mathcal{H}$ is a linear mapping $T:\Dom(T)\subseteq\mathcal{H}\rightarrow \mathcal{H}$, where the domain $\Dom(T)$ is a subspace of $\mathcal{H}$. The range $\Ran(T):=\{Tu : u\in\Dom(T)\}$ and the null-space $\Nul(T):=\{u\in\Dom(T) : Tu=0\}$. Let $\overline{\Ran(T)}$ denote the closure of the range in $\mathcal{H}$. An operator is defined to be \textit{closed} if the graph $\Gra(T):=\{(u,Tu) : u\in\Dom(T)\}$ is a closed subspace of $\mathcal{H}\times \mathcal{H}$, \textit{densely defined} if $\Dom(T)$ is dense in $\mathcal{H}$, and \textit{nilpotent} if $\Ran(\Gamma)\subseteq \Nul(\Gamma)$. The adjoint of a closed, densely defined operator $T$ is denoted by $T^*$. The unital algebra of bounded operators on $\mathcal{H}$ is denoted by $\mathcal{L}(\mathcal{H})$, where the unit is the identity operator $I$ on~$\mathcal{H}$. Given another Hilbert space $\mathcal{K}$, let $\mathcal{L}(\mathcal{H},\mathcal{K})$ denote the space of bounded operators from $\mathcal{H}$ into $\mathcal{K}$.

We now recall the operator-theoretic results obtained by Axelsson, Keith and McIntosh in \cite{AKMc}. Consider three operators $\{\Gamma,B_1,B_2\}$ acting in a Hilbert space $\mathcal{H}$, with norm $\|\cdot\|$ and inner-product $\langle\cdot,\cdot\rangle$, that satisfy the following properties:

\newcounter{Hyp}
\begin{list}
{(H\arabic{Hyp})}
{\setlength{\leftmargin}{.9cm}
\setlength{\rightmargin}{0cm}
\setlength{\topsep}{10pt}
\setlength{\itemsep}{3pt}
\usecounter{Hyp}}
\item The operator $\Gamma:\Dom(\Gamma)\subseteq \mathcal{H}\rightarrow \mathcal{H}$ is densely defined, closed and nilpotent. The condition that $\Gamma$ is nilpotent implies that $\Gamma^2=0$ on $\Dom(\Gamma)$;
\item The operators $B_1$ and $B_2$ are bounded and there exist $\kappa_1,\,\kappa_2>0$ such that the following accretivity conditions are satisfied:
\begin{align*}
\re\langle B_1u,u\rangle &\geq \kappa_1\|u\|^2\quad\text{for all}\quad u\in\Ran(\Gamma^*);\\
\re\langle B_2u,u\rangle &\geq \kappa_2\|u\|^2\quad\text{for all}\quad u\in\Ran(\Gamma).
\end{align*}
The angles of accretivity are then defined as follows:
\begin{align*}
\omega_1 &:= \sup_{u\in\Ran(\Gamma^*) \setminus \{0\}} |\arg\langle B_1u,u\rangle| < \tfrac{\pi}{2}\;\\
\omega_2 &:= \sup_{u\in\Ran(\Gamma) \setminus \{0\}} |\arg\langle B_2u,u\rangle| < \tfrac{\pi}{2}.
\end{align*}
Also, set $\omega:=\tfrac{1}{2}(\omega_1+\omega_2)$.
\item The operators satisfy $\Gamma^*B_2B_1\Gamma^*=0$ on $\Dom(\Gamma^*)$ and $\Gamma B_1B_2\Gamma =0$ on $\Dom(\Gamma)$. This implies that $\Gamma B_1^*B_2^*\Gamma=0$ on $\Dom(\Gamma)$ and $\Gamma^* B_2^*B_1^*\Gamma^* =0$ on $\Dom(\Gamma^*)$.
\end{list}

Now introduce the following operators.
\begin{Def}\label{Def: Pi}
Let $\Pi:=\Gamma+\Gamma^*$, $\Gamma_B:=B_2^*\Gamma B_1^*$ and $\Pi_B:=\Gamma+\Gamma_B^*$.
\end{Def}

Lemma~4.1 and Corollary 4.3 in \cite{AKMc} show that $\Gamma_B^*=B_1\Gamma^*B_2$ and $(\Pi_B)^*=\Gamma^*+\Gamma_B$, that each of these operators is closed and densely defined, and that $\Gamma_B$ and $\Gamma_B^*$ are nilpotent. The following results are from Lemma~4.2 in \cite{AKMc}:
\begin{align}\begin{split}\label{eq: GGB}
\|\Gamma u\| + \|\Gamma_B^*u\| &\eqsim \|\Pi_B u\|\quad\text{for all}\quad u\in\Dom(\Pi_B);\\
\|\Gamma^* u\| + \|\Gamma_Bu\| &\eqsim \|\Pi_B^* u\|\quad\text{for all}\quad u\in\Dom(\Pi_B^*).
\end{split}\end{align}
Proposition~2.2 in \cite{AKMc} establishes the following Hodge decompositions of $\mathcal{H}$:
\[
\mathcal{H} = \Nul(\Pi_B) \oplus \overline{\Ran(\Gamma_B^*)} \oplus \overline{\Ran(\Gamma)} = \Nul(\Pi_B^*) \oplus \overline{\Ran(\Gamma_B)} \oplus \overline{\Ran(\Gamma^*)},
\]
where there is no orthogonality implied by the direct sums (except in the case $B_1=B_2=I$) and the decompositions are topological. It is also shown there that $\Nul(\Pi_B)=\Nul(\Gamma_B^*)\cap\Nul(\Gamma)$ and $\overline{\Ran(\Pi_B)}=\overline{\Ran(\Gamma_B^*)} \oplus \overline{\Ran(\Gamma)}$. Furthermore, Proposition~2.5 in \cite{AKMc} establishes that $\Pi_B$ is $\text{type }S_{\omega}$, which we make precise at the end of this section.

We work within this general framework and consider a complete Riemannian manifold $M$ with geodesic distance $\rho$ and Riemannian measure $\mu$. The covariant derivative $\nabla:C^\infty(T^{k,l}M) \rightarrow C^\infty(T^{k+1,l}M)$ is defined for each $k,\, l\in\N_0$ by extending the Levi-Civita connection on $M$ to smooth tensor fields. For functions $u\in C^\infty(M)$, the smooth covector field $\nabla u$ is defined by $\nabla u (X) := X(u)$ for all $X\in C^\infty(TM)$ (see \eqref{eq: nabla.function} for further details). The space $\mathcal{W}^{1,2}(M)$ consists of all $u$ in $C^\infty(M)$ with
\[
\|u\|_{\mathcal{W}^{1,2}(M)}^2 := \|u\|_{L^2(M)}^2 + \|\nabla u\|_{L^2(T^*M)}^2 <\infty.
\]
The Sobolev space $W^{1,2}(M)$ is then defined to be the completion of $\mathcal{W}^{1,2}(M)$ under the norm $\|\cdot\|_{\mathcal{W}^{1,2}(M)}$. This completion is identified with the subspace of $L^2(M)$ consisting of all $u$ in $L^2(M)$ for which there exists a Cauchy sequence $(u_n)_n$ in $\mathcal{W}^{1,2}(M)$ that converges to $u$ in $L^2(M)$, in which case $\nabla u$ is defined to be the limit of $(\nabla u_n)_n$ in $L^2(T^*M)$ and
\[
\|u\|_{W^{1,2}(M)}^2:= \|u\|_{L^2(M)}^2 + \|\nabla u\|_{L^2(T^*M)}^2.
\]
Further details on this identification are contained in Section~2.2 of~\cite{Hebey1999}.

For each $N\in\N$, define the spaces of $\C^N$-valued functions
\[
\textstyle L^2(M;\C^N) := \bigoplus^N L^2(M) \qquad\text{and}\qquad W^{1,2}(M;\C^N):=\bigoplus^N W^{1,2}(M).
\]
Let $(e_1,\ldots,e_N)$ denote the standard basis of $\C^N$ so that for each $\C^N$-valued function $u$ there exists $N$ unique $\C$-valued functions $u_\alpha$ such that $u=\sum_{\alpha=1}^N u_\alpha e_\alpha$. For each measurable subset $S\subseteq M$ satisfying $0<\mu(S)<\infty$, and each $\C^N$-valued function $u=\sum_{\alpha=1}^N u_\alpha e_\alpha$ for which each $u_\alpha$ is locally integrable, define
\[
\int_S u\ \d\mu : = \sum_{\alpha=1}^N\left(\int_S u_\alpha\ \d\mu\right) e_\alpha
\qquad\text{and}\qquad
u_S:= \dashint_S u\ \d\mu:=\frac{1}{\mu(S)}\int_S u\ \d\mu.
\]
Finally, for each $z=\sum_{\alpha=1}^Nz_\alpha e_\alpha$ in $\C^N$, let $|z|:=\sum_{\alpha=1}^N z_\alpha \overline{z_\alpha}$.

Now consider the following additional hypotheses for the operators $\{\Gamma,B_1,B_2\}$ and the Hilbert space~$\mathcal{H}$, which are analogous to those used by  Axelsson, Keith and McIntosh in \cite{AKMc2}:
\begin{list}
{(H\arabic{Hyp})}
{\setlength{\leftmargin}{.9cm}
\setlength{\rightmargin}{0cm}
\setlength{\topsep}{10pt}
\setlength{\itemsep}{3pt}
\usecounter{Hyp}}\setcounter{Hyp}{3}
\item The Hilbert space $\mathcal{H}=L^2(M;\C^N)$ for some $N\in\N$;

\item The operators $B_1$ and $B_2$ are matrix-valued pointwise multiplication operators in the sense that the functions defined for all $x$ in $M$ by $x\mapsto B_1(x)$ and $x\mapsto B_2(x)$ belong to $L^\infty(M;\mathcal{L}(\C^N))$.

\item The operator $\Gamma$ is a first-order differential operator in the following sense. There exists a constant $C_\Gamma>0$ such that for all $\eta\in C_c^\infty(M)$ we have $\Dom(\Gamma)\subseteq \Dom(\Gamma\circ \eta I)$, where $\eta I$ is the operator of pointwise multiplication by $\eta$, and the commutator $[\Gamma,\eta I]$ is a pointwise multiplication operator satisfying
\[
|[\Gamma,\eta I]u(x)| \leq C_\Gamma\ |\nabla \eta(x)|_{T^*_xM}|u(x)|
\]
for all $u\in\Dom(\Gamma)$ and almost all $x\in M$. This implies that the same hypotheses hold with $\Gamma$ replaced by $\Gamma^*$ and $\Pi$.

\item There exists a constant $c>0$ such that the following hold for all open geodesic balls $B$ contained in $M$ of radius $r\leq1$:
\begin{align*}
\!\!\left|\; \int_B \Gamma u\ \d\mu\, \right| &\leq c \mu(B)^{\frac{1}{2}} \|u\|_{L^2(M;\C^N)} \ \text{for all }u\in\Dom(\Gamma)\text{ compactly supported in }B;\\
\!\!\left|\int_B \Gamma^* u\ \d\mu\right| &\leq c \mu(B)^{\frac{1}{2}} \|u\|_{L^2(M;\C^N)} \ \text{for all } u\in\Dom(\Gamma^*) \text{ compactly supported in } B.
\end{align*}

\item There exists a constant $c>0$ such that
\[
\|u\|_{W^{1,2}(M;\C^N)} \leq c \|\Pi u\|_{L^2(M;\C^N)}
\]
for all $u\in\Ran(\Gamma) \cup \Ran(\Gamma^*) \cap\Dom(\Pi)$.
\end{list}

We consider manifolds that have at most exponential volume growth and on which a local Poincar\'{e} inequality holds. This is made precise below using the following notation. A ball in $M$ will always refer to an open geodesic ball. Given $x\in M$ and $r>0$, let $B(x,r)$ denote the ball in $M$ with centre $x$ and radius $r$, and let $V(x,r)$ denote the Riemannian measure $\mu(B(x,r))$. Given $\alpha,r>0$ and a ball $B$ of radius $r$, let $\alpha B$ denote the ball with the same centre as $B$ and radius $\alpha r$. For all measurable subsets $E,F\subseteq M$, let $\ca_E$ denote the characteristic function of $E$ on $M$, and define $\rho(E,F):=\inf_{x\in E, y\in F} \rho(x,y)$ provided the infimum exists.

\begin{Def}\label{Def: Eloc}
A complete Riemannian manifold $M$ has \textit{exponential volume growth} if there exist constants $c\geq1$ and $\kappa,\lambda\geq0$ such that
\begin{equation}\label{E}\tag{E$_\text{loc}$}
0<V(x,\alpha r)\leq c\alpha^\kappa e^{\lambda \alpha r} V(x,r)<\infty
\end{equation}
for all $\alpha\geq1$, $r>0$ and $x\in M$.
\end{Def}

\begin{Def}\label{Def: Ploc}
A complete Riemannian manifold $M$ satisfies a \textit{local Poincar\'{e} inequality} if there exists a constant $c\geq1$ such that
\begin{equation}\label{P}\tag{P$_\text{loc}$}
\|\ca_B (u-u_B)\|_{L^2(M)}^2 \leq c\,r^2 (\|\ca_B u\|_{L^2(M)}^2 + \|\ca_B \nabla u\|_{L^2(T^*M)}^2)
\end{equation}
for all $u \in W^{1,2}(M)$ and balls $B$ in $M$ of radius $r\leq1$.
\end{Def}

In Section~\ref{Section: KatoApp}, we obtain the solution of the Kato square root problem stated in Theorem~\ref{Thm: MainSubManifoldThm} as a corollary of the following general result.

\begin{Thm}\label{Thm: mainquadraticestimate}
Let $M$ be a complete Riemannian manifold satisfying \eqref{E} and \eqref{P}. Given operators $\{\Gamma,B_1,B_2\}$ on $L^2(M;\C^N)$ satisfying hypotheses $(\mathrm{H}1)-(\mathrm{H}8)$, the perturbed operator $\Pi_B:=\Gamma+B_1\Gamma^*B_2$ satisfies the quadratic estimate
\begin{equation}\label{eq: mainquadraticestimate}
\int_0^\infty \|t\Pi_B(I+t^2\Pi_B^2)^{-1}u\|_2^2\ \frac{\d t}{t} \eqsim \|u\|_2^2
\end{equation}
for all $u$ in $\overline{\Ran(\Pi_B)}$.
\end{Thm}

Theorem~\ref{Thm: mainquadraticestimate} is proved in Section~\ref{Section: MainLQE}. We conclude this section by explaining how it implies a Kato square root estimate for $\Pi_B$. To do this, recall that Proposition~2.5 of~\cite{AKMc} shows that $\Pi_B$ is an operator of $\text{type }S_{\omega}$, where $\omega\in [0,\pi/2)$ is from~(H2). This is defined to mean that the spectrum of $\Pi_B$ is contained in the closed bisector 
\[
S_\omega := \{z\in\C : |\arg z|\leq\omega \text{ or } |\pi-\arg z|\leq\omega\}
\]
and that for each $\theta\in(\omega,\pi/2)$ there exists a constant $C_\theta>0$ such that
\begin{equation}\label{eq: KatoSw}
{\textstyle \sup_{z \in \C\setminus S_{\theta}}}\, |z| \|(z I-\Pi_B)^{-1}\| \leq C_\theta
\end{equation}

The theory of type $S_{\omega}$ operators is well-understood and can be found in, for instance, \cite{Kato,McIntosh1986,Haase2006}. Let us briefly mention the McIntosh functional calculus from~\cite{McIntosh1986}. For all $\theta\in(0,\pi/2)$, let $H^\infty(S_{\theta}^o)$ denote the algebra of bounded holomorphic functions on the open bisector
$
S_\theta^o := \{z\in\C\setminus\{0\} : |\arg z|<\theta \text{ or } |\pi-\arg z|<\theta\}.
$
Also, let $\Psi(S_\theta^o)$ denote the subspace of functions $\psi$ in $H^\infty(S^o_\theta)$ for which there exists $\alpha,\beta>0$ such that $|\psi(z)| \lesssim \min\{|z|^\alpha,|z|^{-\beta}\}$
for all $z\in S^o_\theta$. The resolvent bounds in \eqref{eq: KatoSw} and the decay properties of $\psi\in\Psi(S^o_\theta)$ allow one to use the Cauchy integral formula to define the bounded operator
$
\psi(\Pi_B)u: = \frac{1}{2\pi i} \int_\gamma \psi(z) (z I-\Pi_B)^{-1}u\ \d z
$
for all $u\in L^2(M;\C^N)$, where $\gamma$ is the positively oriented boundary of $S^o_\mu$ for any $\mu\in(\omega,\theta)$.

The operator $\Pi_B$ is said to have a bounded $H^\infty(S_{\theta}^o)$ functional calculus in the space $L^2(M;\C^N)$ if for each ${\theta \in (\omega,\pi/2)}$, there exists $c>0$ such that $\|\psi(\Pi_B)\| \leq c \|\psi\|_\infty$ for all $\psi\in \Psi(S^o_\theta)$. Given $f\in H^\infty(S^o_\theta)$, this property allows one to define the bounded operator $f(\Pi_B)$ on $L^2(M;\C^N)$ by $f(\Pi_B)u = f(0)\mathsf{P}_{\Nul(\Pi_B)}u+\lim_{n\rightarrow\infty} \psi_n(\Pi_B)u$ for all $u\in L^2(M;\C^N)$, where $\mathsf{P}_{\Nul(\Pi_B)}$ is the projection from $L^2(M;\C^N)$ onto $\Nul(\Pi_B) $ and $(\psi_n)_n$ is any sequence in $\Psi(S^o_\theta)$ converging to $f$ uniformly on compact subsets of $S^o_\theta$.

An essential feature of the McIntosh functional calculus is that quadratic estimates such as those in \eqref{eq: mainquadraticestimate} are equivalent to the property that $\Pi_B$ has a bounded $H^\infty(S^o_\theta)$ functional calculus. The following Kato square root estimate is then obtained as a corollary of Theorem~\ref{Thm: mainquadraticestimate} by defining $f$ in $H^\infty(S^o_\theta)$ as $f(z)=\sgn(\re (z)) = z/\sqrt{z^2}$ and considering the bounded operator $f(\Pi_B)=\Pi_B/\sqrt{\Pi_B^2}$. The arguments can be found in more detail above Corollary~2.11 in~\cite{AKMc}.

\begin{Cor}\label{Cor: KatoSQRT}
Assume the hypotheses stated in Theorem~\ref{Thm: mainquadraticestimate}. We then have $\Dom(\sqrt{\Pi_B^2}) = \Dom(\Pi_B) = \Dom(\Gamma)\cap\Dom(\Gamma_B^*)$ with
\[
\|\textstyle{\sqrt{\Pi_B^2}} u\|_2 \eqsim \|\Pi_B u\|_2 \eqsim \|\Gamma u\|_2 + \|\Gamma_B^*u\|_2
\]
for all $u\in\Dom(\sqrt{\Pi_B^2})$.
\end{Cor}

\section{The Solution of the Kato Square Root Problem on Submanifolds}\label{Section: KatoApp}
We now prove Theorem~\ref{Thm: MainSubManifoldThm} as a corollary of Theorem~\ref{Thm: mainquadraticestimate}. Let us first fix notation and dispense with some technicalities of submanifold geometry.

Fix positive integers $n\geq m$ and suppose that $M$ is an $m$-dimensional complete Riemannian submanifold of $\R^n$. This is defined to mean that $M$ is an $m$-dimensional complete Riemannian manifold and that there is a smooth embedding $\iota:M\rightarrow\R^n$. An embedding here refers to an injective immersion that induces the Riemannian metric on $M$ from the ambient Euclidean metric.

A local coordinate chart at $x\in M$ refers to an open set $U\subseteq M$ containing $x$ and a diffeomorphism $\varphi:U\rightarrow \R^m$. The tangent space $T_xM$ is defined to be the space of derivations on the algebra of germs of smooth functions on $M$ at $x$. The derivations $(\partial_1,\ldots,\partial_m)$ defined by
$
\partial_i f = \frac{\partial(f\circ \varphi^{-1})}{\partial x_i}(\varphi(x))
$
for all $f\in C^\infty(V)$, where $V\subseteq M$ is any open set containing $x$, form a basis of the tangent space $T_xM$. The global coordinate chart for $\R^n$ provides an isomorphism between $T_x\R^n$  and $\R^n$ for all $x\in\R^n$. The standard basis of $\R^n$ is denoted by $(e_1,\ldots,e_n)$. When working in either of these bases, we adopt the convention whereby repeated indices are summed over the dimension of the space.

For each $x\in M$, the differential of the embedding $\iota_*:T_xM\rightarrow T_{\iota(x)}\R^n$ is defined for all $v\in T_xM$ by the requirement that
\begin{equation}\label{eq: i*def0}
(\iota_*v)(f) = v(f\circ \iota)
\end{equation}
for all $f$ in the algebra of germs o  smooth functions on $\R^n$ at $\iota(x)$. The isomorphism  $T_{\iota(x)}\R^n\cong\R^n$ allows us to regard this as the mapping ${\iota_*:T_xM\rightarrow \R^n}$ defined by
\begin{equation}\label{eq: i*def}
\iota_*v := v(\iota^\alpha) e_\alpha
\end{equation}
for all $v\in T_xM$, where $\iota(x)=\iota^\alpha(x)e_\alpha$.

The Euclidean metric $\langle x,y \rangle_{\R^n} = \langle x^\alpha e_\alpha, y^\beta e_\beta \rangle_{\R^n} := x^\alpha y^\alpha$ for all $x,\,y\in\R^n$. To say that $\iota$ is an embedding then means that $\iota$ is injective and that for each $x\in M$, the differential $\iota_*$ is injective with the property that
\begin{equation}\label{eq: induced.metric}
\langle u,v \rangle_{T_xM} = \langle \iota_*u,\iota_*v \rangle_{\R^n}
\end{equation}
for all $u,v\in T_xM$. In particular, given a local coordinate chart at $x\in M$, we have
\begin{equation}\label{eq: metric}
g_{ij}(x) := \langle \partial_i, \partial_j \rangle_{T_xM}
= \partial_i\iota^\alpha \partial_j \iota^\alpha.
\end{equation}
The basis $(dx^1,\ldots,dx^m)$ of the cotangent space $T^*_xM$ is defined by requiring that $dx^i(\partial_j)=\delta^i_j$, where $\delta^i_j$ is the Kronecker delta, and then $g^{ij}(x):=\langle dx^i,dx^j\rangle_{T_x^*M}$.

For each $x\in M$, the properties of the embedding guarantee that $\iota_*(T_xM)$ is an $m$-dimensional subspace of $\R^n$. Let $N_{\iota(x)}M$ denote the orthogonal complement of $\iota_*(T_xM)$ in $\R^n$ so that $\iota_*(T_xM) \oplus N_{\iota(x)}M = \R^n$. The normal bundle $NM$ is the bundle over $M$ whose fiber at each $x$ in $M$ is the space $N_{\iota(x)}M$. For each $x\in M$, let $\pi_{\iota_*(TM)}$ denote the orthogonal projection from $\R^n$ onto $\iota_*(T_xM)$ and define $\pi:=\iota_*^{-1}\pi_{\iota_*(TM)}$. The operator $\pi:\R^n\rightarrow T_xM$  is the adjoint of  $\iota_*$, since
\begin{equation}\label{eq: iota.pi.adjoint}
\langle \iota_*u, v \rangle_{\R^n}
=\langle \iota_*u, \pi_{\iota_*(TM)} v \rangle_{\R^n}
=\langle \iota_*u, \iota_* \pi v \rangle_{\R^n}
=\langle u, \pi v \rangle_{T_xM}
\end{equation}
for all $u\in T_xM$ and $v\in\R^n$.

The Euclidean connection on $\R^n$ is the directional derivative $\nabla^{\R^n}_X Y  := X(Y^\alpha)e_\alpha$ for all $X\in C^\infty(T\R^n)$ and $Y=Y^\alpha e_\alpha\in C^\infty(T\R^n)$. The Levi-Civita connection $\nabla^M: C^\infty(TM) \times C^\infty(TM) \rightarrow C^\infty(TM)$ is denoted by $\nabla^M_X Y := \nabla^M (X,Y)$ for all $X,\,Y\in C^\infty(TM)$. It is completely determined by the Christoffel symbols $\Gamma_{ij}^k$, which are the smooth functions satisfying $\nabla^M_{\partial_i} \partial_j = \Gamma_{ij}^k \partial_k$ in a local coordinate chart.

The properties of the embedding guarantee that for each $X\in C^\infty(TM)$ there exists $\widetilde X \in C^\infty(T\R^n)$, which is not necessarily unique, that is an extension of $X$ in the sense that the restriction of $\widetilde X$ to $\iota(M)$ is $\iota_*X$. The second fundamental form $h:C^\infty(TM) \times C^\infty(TM) \rightarrow C^\infty(NM)$ is then defined by
\[
h(X,Y):=\pi_{NM}\big(\nabla^{\R^n}_{\widetilde{X}} \widetilde{Y}\big)
\]
for all $X,\,Y\in C^\infty(TM)$, where $\widetilde{X},\,\widetilde{Y}\in C^\infty(T\R^n)$ denote extensions of $X,\,Y$. It is a standard fact (see, for instance, Lemma~8.1 and Theorem~8.2 in \cite{Lee}) that this definition is independent of the extensions $\widetilde X,\, \widetilde Y$, and that $\iota_* \big(\nabla^M_X Y\big)= \pi_{\iota_*(TM)} \big(\nabla^{\R^n}_{\widetilde{X}} \widetilde{Y}\big)$. This shows that
\[
h(X,Y)=\nabla^{\R^n}_{\widetilde{X}} \widetilde{Y} - \pi_{\iota_*(TM)} \big(\nabla^{\R^n}_{\widetilde{X}} \widetilde{Y}\big)
=\nabla^{\R^n}_{\widetilde{X}} \widetilde{Y} - \iota_* \big(\nabla^M_X Y\big).
\]
In a local coordinate chart, we then have
\begin{align}\begin{split}\label{eq: II.loc.cord}
h_{ij} :&= h(\partial_i,\partial_j) \\
&=\nabla^{\R^n}_{\widetilde\partial_i} \widetilde{\partial_j} - \iota_* \big(\nabla^M_{\partial_i} \partial_j \big)\\
&=(\iota_*\partial_i)(\widetilde{\partial_j}^\alpha) e_\alpha - \iota_* \big(\nabla^M_{\partial_i} \partial_j \big)\\
&= \partial_i(\widetilde{\partial_j}^\alpha\circ\iota) e_\alpha - (\nabla^M_{\partial_i} \partial_j)(\iota^\alpha)e_\alpha \\
&= \partial_i((\iota_*\partial_j)^\alpha) e_\alpha - \Gamma^k_{ij}\partial_k\iota^\alpha e_\alpha \\
&= (\partial_i\partial_j\iota^\alpha - \Gamma^k_{ij}\partial_k\iota^\alpha) e_\alpha \\
&=: h^\alpha_{ij} e_\alpha,
\end{split}\end{align}
where the third equality is obtained by writing $\widetilde{\partial_j} = \widetilde{\partial_j}^\alpha e_\alpha$, applying the ambient connection and the fact that $\widetilde{\partial_i}$ is an extension of $\partial_i$, the fourth uses \eqref{eq: i*def0} and \eqref{eq: i*def}, the fifth equality is obtained by using the fact that $\widetilde{\partial_j}$ is an extension of $\partial_j$ and applying the submanifold connection, and the sixth uses \eqref{eq: i*def}. In a local coordinate chart at $x\in M$, we have $|h|_{T^*_xM\otimes T_x^*M\otimes N_xM} = g^{ij}g^{kl}h_{ik}^\alpha h_{jl}^\alpha$, and the second fundamental form is said to be bounded when 
\[
|h|:=\sup_{x\in M} |h|_{T^*_xM\otimes T_x^*M\otimes N_xM} < \infty.
\]

The covariant derivative $\nabla:C^\infty(T^{k,l}M) \rightarrow C^\infty(T^{k+1,l}M)$ is defined for each $k,\, l\in\N_0$ by extending the Levi-Civita connection on $M$ to smooth tensor fields. For functions $u\in C^\infty(M)$, the smooth covector field $\nabla u$ is defined by
\begin{equation}\label{eq: nabla.function}
\nabla u(dx^i) := \partial_i u.
\end{equation}
The Hessian is the smooth (2,0)-tensor field $\nabla^2u:=\nabla(\nabla u)$ defined by 
\begin{equation}\label{eq: Hessian}
\nabla^2 u(dx^i,dx^j) := \partial_i \partial_j u - \Gamma_{ij}^k \partial_ku.
\end{equation}

The gradient operator $\grad: W^{1,2}(M) \subseteq L^2(M) \rightarrow L^2(TM)$ is defined for each $u\in W^{1,2}(M)$ by the requirement that
\begin{equation}\label{eq: grad.diff}
\langle \grad u, X\rangle_{TM} = \nabla u(X)
\end{equation}
for all $X\in TM$, where $\nabla u$ is defined by the construction of $W^{1,2}(M)$ in Section~\ref{Section: DiracTypeOperators}. The Riesz representation theorem guarantees that the gradient operator is well-defined. If $u\in C^\infty(M)$, then $\nabla u$ is defined by \eqref{eq: nabla.function} and in a local coordinate chart we have $ \grad u = g^{ij} \partial_i u \partial_j$. The Riemannian measure is a Radon measure, so Urysohn's Lemma implies that $C^\infty_c(M)$ is dense in $L^2(M)$ (see, for instance, Proposition~7.9 in \cite{Folland1999}). Therefore, the gradient operator is densely defined and its adjoint $\grad^*$ is defined. The divergence operator $\div: \Dom(\div) \subseteq L^2(TM) \rightarrow L^2(M)$ is defined by $\div:=-\grad^*$, so we have 
\[
\Dom(\div)=\{X \in L^2(TM) : \textstyle{\sup_{u\in W^{1,2}(M); \|u\|_2=1}} |\langle \grad u, X\rangle_{L^2(TM)}| < \infty\}
\]
and
\begin{equation}\label{eq: div.grad.adjoint}
\langle \grad u, X\rangle_{L^2(TM)}  = \langle u, -\div X \rangle_{L^2(M)}
\end{equation}
for all $u\in W^{1,2}(M)$ and $X\in\Dom(\div)$.

The minus sign is inserted in the definition of $\div$ to relate it to the Riemannian divergence $\Div: C^\infty(TM)\rightarrow C^\infty(M)$, which is defined by
\begin{equation}\label{eq: div.def}
\Div X := \tr(\nabla X) := \nabla X(\partial_i,dx^i)
\end{equation}
for all $X\in C^\infty(TM)$. The Riemannian divergence theorem (see, for instance, equation~(III.7.5) in \cite{Chavel2006}) guarantees that $\int_M \Div X\, \d\mu = 0$ for all $X\in C^\infty_c(TM)$.
Then, since $\Div (uX) = \langle \grad u, X\rangle_{TM} + u\Div X$, the integration by parts formula
\begin{equation}\label{eq: intbyparts}
\langle \grad u, X\rangle_{L^2(TM)}  = \langle u, -\Div X \rangle_{L^2(M)}
\end{equation}
is valid for all $u\in C^\infty(M)$ and $X\in C^\infty(TM)$ provided that at least one of $X$ or $u$ is compactly supported.

To prove that $\div$ and $\Div$ coincide on $C^\infty_c(TM)$, we use the density of $C^\infty_c(M)$ in $W^{1,2}(M)$, which was proved on a complete Riemannian manifold by Aubin in \cite{Aubin1976} (see also Theorem~3.1 in \cite{Hebey1999}). Given $X\in C^\infty_c(TM)$ and $u\in W^{1,2}(M)$, let $(u_n)_n$ denote a sequence in $C^\infty_c(M)$ that converges to $u$ in $W^{1,2}(M)$. Using \eqref{eq: intbyparts}, we have
\begin{align*}
|\langle \grad u, X \rangle_{L^2(TM)}| 
&= |\langle\grad (u-u_n), X \rangle_{L^2(TM)} + \langle u_n, -\Div X \rangle_{L^2(M)}| \\
&\leq \|\nabla(u-u_n)\|_{L^2(T^*M)} \|X\|_{L^2(TM)}+ \|u_n\|_{L^2(M)}\|\Div X\|_{L^2(M)} \\
&\leq c_X (\|u-u_n\|_{W^{1,2}(M)} + \|u\|_{L^2(M)}),
\end{align*}
where $c_X:= \|X\|_{L^2(TM)} + \|\Div X\|_{L^2(M)} <\infty$ because $X\in C^\infty_c(TM)$. This shows that $C^\infty_c(TM) \subseteq \Dom(\div)$ and a similar argument shows that $\Div X = \div X$.

The Laplacian $\Delta:C^\infty(M) \rightarrow C^\infty(M)$ is defined by $\Delta u := -\Div \grad u$ for all ${u\in C^\infty(M)}$, with $\grad u \in C^\infty(TM)$ defined by \eqref{eq: grad.diff}, so we have $\Delta v = -\div \grad v$ for all $v\in C_c^\infty(M)$. This completes the setup required to prove Theorem~\ref{Thm: MainSubManifoldThm}.

\begin{proof}[Proof of Theorem~\ref{Thm: MainSubManifoldThm}]
Suppose that $a\in L^\infty(M)$, $A_{00}\in L^\infty(M)$, $A_{10}\in L^\infty(TM)$, $A_{01}\in L^\infty(T^*M)$ and $A_{11}\in L^\infty(\End(TM))$ satisfy the accretivity conditions in~\eqref{eq: accr.a.A}. We now define operators $\{\Gamma,B_1,B_2\}$ acting in a Hilbert space $\mathcal{H}$ such that the divergence form operator $L_{A,a}$ in \eqref{eq: LA} is a component of the first-order system $\Pi_B^2:=\Gamma+B_1\Gamma^*B_2$.

The operator $A$ from \eqref{eq: A.Def} is used to define the pointwise multiplication operator $\tilde A \in L^\infty(M;\mathcal{L}(\C^{1+n}))$ by
\[
\tilde A(x) :=
\begin{bmatrix}1 & 0 \\0 &(\iota_*)_x\end{bmatrix}
\begin{bmatrix}(A_{00})_x & (A_{01})_x \\(A_{10})_x & (A_{11})_x\end{bmatrix}
\begin{bmatrix}1 & 0 \\0 & \pi_x \end{bmatrix}
\]
for almost all $x\in M$. The components $\tilde A_{00} \in L^\infty(M;\mathcal{L}(\C))$, $\tilde A_{01} \in L^\infty(M;\mathcal{L}(\C^n;\C))$, $\tilde A_{10} \in L^\infty(M;\mathcal{L}(\C;\C^n))$ and $\tilde A_{11} \in L^\infty(M;\mathcal{L}(\C^n))$ satisfy
\[
\tilde A(x)= \begin{bmatrix}\tilde A_{00}(x) & \tilde A_{01}(x) \\\tilde A_{10}(x) & \tilde A_{11}(x)\end{bmatrix}
=\begin{bmatrix}(A_{00})_x & (A_{01})_x\pi_x \\ (\iota_*)_x(A_{10})_x & (\iota_*)_x(A_{11})_x\pi_x\end{bmatrix}.
\]
The following diagram now commutes, where $I$ denotes the identity on $L^2(M)$:
\[\xymatrix{
L^2(M) \ar[d]_-{L_A} \ar[r]^-{I\brack\grad}
&L^2(M)\oplus L^2(TM) \ar[d]^-{A} \ar[r]^-{{\!\!I\ 0}\brack{0\ \iota_*}}
&L^2(M;\C^{1+n}) \ar[d]^{\tilde{A}}\\
L^2(M)
&L^2(M)\oplus L^2(TM) \ar[l]_-{a[I, -\div]}
&L^2(M;\C^{1+n}) \ar[l]_-{{I\ 0}\brack{0\ \pi}}
}\]
Following \cite{AKMc2}, define the operator
\[
S := \begin{bmatrix} I   & 0\\ 0   &\iota_* \\ \end{bmatrix} 
\begin{bmatrix} I \\ \grad \\ \end{bmatrix}
= \begin{bmatrix} I \\ \iota_*\grad \\ \end{bmatrix}
: \Dom(S) \subseteq L^2(M) \rightarrow L^2(M; \C^{1+n})
\]
with $\Dom(S)=W^{1,2}(M)$ and adjoint
\[
S^* = \begin{bmatrix} I   & -\div \end{bmatrix}
\begin{bmatrix} I   &0\\ 0   &\pi \\ \end{bmatrix}
= \begin{bmatrix} I   & -\div \pi \end{bmatrix}
: \Dom(S^*) \subseteq L^2(M; \C^{1+n}) \rightarrow L^2(M).
\]
These are closed and densely defined. The operators $\{\Gamma,B_1,B_2\}$ acting in the Hilbert space $\mathcal{H} = L^2(M)\oplus L^2(M;\C^{1+n})$ are now defined below:
\begin{equation}\label{eq: GamB1B2def}
\Gamma =
\begin{bmatrix} 0   &0\\ S   &0 \\ \end{bmatrix};
\quad
\Gamma^* = \begin{bmatrix} 0   & S^*\\ 0   &0 \\ \end{bmatrix};
\quad
B_1 = \begin{bmatrix} a   &0\\ 0     &0 \\ \end{bmatrix};
\quad
B_2 = \begin{bmatrix} 0   &0\\ 0   &\tilde A \\ \end{bmatrix}.
\end{equation}
In that case, the operators from Definition~\ref{Def: Pi} are as follows:
\begin{align*}
\Gamma_B^* & = B_1\Gamma^*B_2 = \begin{bmatrix} 0 &aS^*\tilde A\\ 0 &0 \\ \end{bmatrix};
\quad
\Pi_B = \Gamma + \Gamma_B^*
= \begin{bmatrix} 0 &aS^*\tilde A\\ S &0 \\ \end{bmatrix}; \\
\Pi_B^2 &= \begin{bmatrix} aS^*\tilde AS & 0 \\ 0 & SaS^*\tilde A \\ \end{bmatrix}
= \begin{bmatrix} L_A & 0 \\ 0 & SaS^*\tilde A \\ \end{bmatrix}.
\end{align*}

The assumption that $M$ has a bounded second fundamental form implies a lower bound on its Ricci curvature. This is proved in Lemma~\ref{Lem: II.bdd.prop} below. The lower bound on Ricci curvature implies that both \eqref{E} and \eqref{P} are satisfied on $M$. The volume growth condition \eqref{E} is a consequence of the Bishop--Gromov volume comparison theorem (see, for instance, \cite{Chavel2006}). The local Poincar\'{e} inequality is a result of Buser in \cite{Buser1982}. A concise summary of these and other properties of manifolds with Ricci curvature bounded below can be found in Section~5.6.3 of \cite{Saloff2002}.

The assumption that $M$ has a bounded second fundamental form also implies that the operators $\{\Gamma,B_1,B_2\}$ on $\mathcal{H}$ from \eqref{eq: GamB1B2def} satisfy hypotheses (H1)--(H8) from Section~\ref{Section: DiracTypeOperators}. This is proved in Proposition~\ref{Prop: MainSubManifoldLQE} below.

Therefore, the requirements of Theorem~\ref{Thm: mainquadraticestimate} are satisfied and Corollary~\ref{Cor: KatoSQRT} implies that
$\Dom(\sqrt{\Pi_B^2})=\Dom(\Pi_B)=\Dom(\Gamma)\cap\Dom(\Gamma_B^*)$ with
\[
\|\textstyle{\sqrt{\Pi_B^2}}u\|_2 \eqsim \|\Pi_B u\|_2 \eqsim \|\Gamma u\|_2 + \|\Gamma_B^*u\|_2
\]
for all $u\in \Dom(\sqrt{\Pi_B^2})$. When we restrict this result to $u\in L^2(M)$, we obtain $\Dom(\sqrt{L_{A,a}}) = \Dom(S) = W^{1,2}(M)$ with
\[
\|\textstyle{\sqrt{L_{A,a}}}u\|_2 \eqsim \|Su\|_2 = \|u\|_{W^{1,2}(M)}
\]
for all $u\in W^{1,2}(M)$, as required.
\end{proof} 

It remains to prove the two claims about a submanifold with bounded second fundamental form that were made in the proof above. This is the content of Lemma~\ref{Lem: II.bdd.prop}, which records some of the geometric properties implied by a bounded second fundamental form, and Proposition~\ref{Prop: MainSubManifoldLQE}.

\begin{Lem}\label{Lem: II.bdd.prop}
Let $M$ be a complete Riemannian submanifold of $\R^n$. If the second fundamental form $h$ satisfies $|h|:= \sup_{x\in M} |h|_{T^*_xM\otimes T^*_xM\otimes N_xM} < \infty$, then the Ricci curvature of $M$ is bounded below and the injectivity radius of $M$ is positive.
\end{Lem}

\begin{proof}
The Gauss equation (see, for instance, Theorem~8.4 in \cite{Lee}) shows that a bound on the second fundamental form implies a bound on the magnitude of any sectional curvature, and hence a lower bound on the Ricci curvature of~$M$.

The refined version of Klingenberg's formula presented by Abresch and Meyer in Lemma 1.8 of \cite{AbreschMeyer1997} (see also Lemma 5.6 in \cite{CheegerEbin1975}) shows that the injectivity radius of~$M$ is equal to $\min\{\conj M, \tfrac{1}{2} \inf_{x\in M} \ell_M(x)\}$, where $\conj M$ is the conjugate radius of~$M$ and $\ell_M(x)$ is the length of the shortest nontrivial geodesic loop $\gamma:[0,1]\rightarrow M$ starting and ending at $x\in M$. The sectional curvature bounds mentioned previously guarantee that $\conj M$ is positive (see, for instance, Corollary~B.21 in \cite{ChowKnopf2004}). Now suppose that $\gamma$ is a geodesic loop as above. The curvature of $\gamma$ is equal to $h(\gamma',\gamma')$ (see, for instance, Lemma~8.5 in~\cite{Lee}). Fenchel \cite{Fenchel1929} and Buser~\cite{Borsuk1947} proved that the total curvature of a closed curve in $\R^n$ is greater than or equal to $2\pi$, hence we have
$2\pi \leq \int |h(\gamma',\gamma')|\ \d s \leq |h| \ell_M(x)$ for all $x\in M$, and the proof is complete.
\end{proof}

The proof of the following proposition completes the proof of Thereom~\ref{Thm: MainSubManifoldThm}.

\begin{Prop}\label{Prop: MainSubManifoldLQE}
Let $M$ be a complete Riemannian submanifold of $\R^n$. If the second fundamental form $h$ satisfies $|h|:= \sup_{x\in M} |h|_{T^*_xM\otimes T^*_xM\otimes N_xM} < \infty$, then the operators $\{\Gamma,B_1,B_2\}$ on the Hilbert space $\mathcal{H} = L^2(M;\C^{2+n})$ from \eqref{eq: GamB1B2def} satisfy hypotheses (H1)--(H8) from Section~\ref{Section: DiracTypeOperators}.
\end{Prop}

\begin{proof}
Let $\|\cdot\|$ and $\langle\cdot,\cdot\rangle$ denote the norm and inner-product on the space $L^2(M;\C^{2+n})$. Hypotheses (H1) and (H3)--(H6) are immediate and do not require the geometric assumptions in the proposition.

(H2).  There are two estimates to prove:

\noindent(i) If $u\in \Ran(\Gamma^*)$, then $u=(S^* \tilde u, 0)$ for some $\tilde u \in \Dom(S^*)$ such that $S^* \tilde u \in L^2(M)$. The accretivity assumption on $a$ in \eqref{eq: accr.a.A} then implies that
\[
\re\langle B_1u,u\rangle = \re\langle a S^* \tilde u, S^* \tilde u\rangle \geq \kappa_1\| S^* \tilde u\|^2 = \kappa_1\|u\|^2.
\]

\noindent(ii) If $u\in\Ran(\Gamma)$, then $u=(0, Su_0)$ for some $u_0 \in \Dom(S) = W^{1,2}(M)$. The duality in \eqref{eq: iota.pi.adjoint} and the accretivity assumption on $A$ in \eqref{eq: accr.a.A} then imply that
\begin{align*}
\re\langle B_2u,u\rangle &= \re\langle \begin{bmatrix} I   & 0\\ 0   &\iota_* \\ \end{bmatrix} A \begin{bmatrix} I   &0\\ 0   &\pi \\ \end{bmatrix} Su_0, Su_0\rangle \\
&= \re\langle A \begin{bmatrix} I   \\  \grad \\ \end{bmatrix} u_0,  \begin{bmatrix} I   \\  \grad \\ \end{bmatrix} u_0\rangle_{L^2(M)\oplus L^2(TM)} \\
&\geq \kappa_2\|u_0\|_{W^{1,2}(M)}^2 \\
&= \kappa_2\|u\|^2.
\end{align*}

(H7). There are two estimates to prove:

\noindent(i) To prove the first estimate in (H7), it suffices to show that there exists $c>0$ such that for all balls $B$ in $M$, the following hold for all $u\in W^{1,2}(M)$ with compact support in $B$:
\[
\left|\; \int_B u\ \d\mu\, \right| \leq c\mu(B)^{\frac{1}{2}} \|u\|_{L^2(M)};\quad
\left|\; \int_B \iota_* \grad u\ \d\mu\, \right| \leq c\mu(B)^{\frac{1}{2}} \|u\|_{L^2(M)}
\]
The first of these is given by the Cauchy--Schwartz inequality. To prove the second, we start by showing that
\begin{equation}\label{eq: intbyparts2}
\langle \grad v, X \rangle_{L^2(TM)}  =  \langle v, -\Div X \rangle_{L^2(M)}
\end{equation}
for all $v\in W^{1,2}(M)$ that are compactly supported and all $X\in C^\infty(M)$. Suppose that $v\in W^{1,2}(M)$ is supported in a compact set $K\subseteq M$ and choose a sequence $(v_n)_n$ of functions in $\mathcal{W}^{1,2}(M)$ supported in $K$ that converge to $v$ in $W^{1,2}(M)$. For each $n\in\N$, the integration by parts formula \eqref{eq: intbyparts} shows that
\begin{align*}
|\langle \grad v&, X \rangle_{L^2(TM)}  -  \langle v, -\Div X \rangle_{L^2(M)}| \\
&= |\langle\grad (v-v_n), X \rangle_{L^2(TM)} - \langle (v-v_n), -\Div X \rangle_{L^2(M)}| \\
&\leq \|\grad(v-v_n)\|_{L^2(TM)} \|\ca_K X\|_{L^2(TM)}+ \|v-v_n\|_{L^2(M)}\|\ca_K\Div X\|_{L^2(M)}.
\end{align*}
The smoothness of $X$ guarantees that both $\|\ca_K X\|_{L^2(TM)}$ and $\|\ca_K\Div X\|_{L^2(M)}$ are finite, so the bound above can be made arbitrarily small to prove \eqref{eq: intbyparts2}. 

We now obtain
\begin{align*}
\left|\int_B \iota_* \grad u\ \d\mu \right| &= \left|\int \grad u\,( \iota^\alpha)\ \d\mu \ e_\alpha\right|\\
&= \left|\int \langle \grad \iota^\alpha, \grad u \rangle_{TM} \ \d\mu \ e_\alpha \right| \\
&= \left|\int u \Delta \iota^\alpha\ \d\mu \ e_\alpha \right| \\
&\leq \sup_{x\in M} |\Delta \iota^\alpha(x)e_\alpha| \int_B |u| \ \d\mu\\
&\leq |h| \mu(B)^{1/2} \|u\|_{L^2(M)},
\end{align*}
where the first equality is given by the definition of $\iota_*$ in \eqref{eq: i*def}, the second is given by the definition of the gradient in \eqref{eq: grad.diff} for the smooth function $\iota^\alpha$, and the third is given by \eqref{eq: intbyparts2} because $u$ has compact support. The final inequality is given by the bound on the second fundamental form $h$ and the Cauchy--Schwartz inequality. This estimate actually only requires a bound on the mean curvature $H:=\tr_g h := g^{ij}h_{ij}=\Delta \iota^\alpha e_\alpha$, where the last equality can be seen by considering the local coordinate expression \eqref{eq: II.loc.cord} for $h$ in geodesic normal coordinates.\\

\noindent(ii) To prove the second estimate in (H7), we start by showing that
\begin{equation}\label{eq: div.thm.2}
\int \div X\ \d\mu = 0
\end{equation}
for all $X\in\Dom(\div)$ that are compactly supported. To this end, let us verify that $\div$ preserves support. Suppose that $X\in\Dom(\div)$ is supported in a compact set $K$ and choose an open set $U\subseteq M$ that contains $K$. Also, choose $\eta \in C^\infty(M)$ such that $\eta=1$ on $K$ and $\eta=0$ on $M\setminus U$. For all $v\in C_c^\infty(M)$, we have
\begin{align*}
|\langle \div X - \eta \div X, v \rangle_{L^2(M)}| 
&=|\langle [\div,\eta I] X, v \rangle_{L^2(M)}| \\
&=|\langle X, [\grad,\eta I] v \rangle_{L^2(TM)}| \\
&\leq \|X\|_{L^2(TM)} \left( \int_K |\grad \eta(x)|_{T_xM}^2 |v(x)|^2\, \d\mu(x)\right)^{\frac{1}{2}} \\
&=0,
\end{align*}
since $\eta$ is constant on $K$. This shows that $\|\div X - \eta \div X\|_{L^2(M)}=0$, and hence $\div X$ is supported in $U$. The above construction applies to an arbitrary open set $U$ that contains $K$, so we conclude that $\div X$ is supported in $K$. Now fix $\eta \in C^\infty(M)$ as above and choose a sequence $(X_n)_n$ of vector fields in $C^\infty(TM)$ supported in $K$ that converge to $X$ in $L^2(TM)$. For each $n\in\N$, the vector field $X_n\in C^\infty_c(TM)$, hence $\div X_n = \Div X_n$ and by the Riemannian divergence theorem we have
\begin{align*}
\left|\int \div X\ \d\mu\right|
&= \left|\int \div (X-X_n)\ \d\mu\right| \\
&= \left|\int \eta \div (X-X_n)\ \d\mu\right| \\
&= \left| \langle \grad \eta, X-X_n \rangle_{L^2(TM)} \right| \\
&\leq \|\ca_K\grad\eta\|_{L^2(TM)} \|X-X_n\|_{L^2(TM)}.
\end{align*}
The smoothness of $\eta$ guarantees that $\|\ca_K\grad\eta\|_{L^2(TM)}$ is finite, so the bound above can be made arbitrarily small to prove \eqref{eq: div.thm.2}.

Now suppose that $u=(u_0,u_1,\tilde u) \in L^2(M;\C^{1+1+n})$ has compact support in a ball $B$ in $M$ and that $(u_1,\tilde u)\in \Dom(S^*)$. This implies that $\pi \tilde u \in \Dom(\div)$, and since $\pi$ is defined pointwise on $M$, the vector field $\pi \tilde u$ is compactly supported in $B$. Therefore, using the Cauchy--Schwartz inequality and \eqref{eq: div.thm.2}, we obtain
\begin{align*}
\left|\int_B \Gamma^* u\ \d\mu\right|
&= \left|\int_B u_1 - \div\pi\, \tilde u\ \d\mu\right| \\
&\leq \mu(B)^{\frac{1}{2}}\|u_1\|_{L^2(M)} + \left|\int \div (\pi\tilde  u)\ \d\mu\right| \\
&= \mu(B)^{\frac{1}{2}}\|u_1\|_{L^2(M)} \\
&\leq \mu(B)^{\frac{1}{2}} \|u\|.
\end{align*}
This completes the proof of the second estimate in (H7). Note that we did not require the condition that the radius $r(B)\leq 1$ to verify (H7).\\

(H8). Consider two cases:

\noindent(i) Let $u\in \Ran(\Gamma^*)\cap\Dom(\Pi)$. This implies that $u=(u_0, 0)$ for some $u_0 \in L^2(M)$ and
\[
\|\Pi u\| = \|\Gamma u\|
= \|Su_0\|_{L^2(M;\C^{1+n})}
= \|u\|_{W^{1,2}(M;\C^{2+n})},
\]
as required.\\

\noindent(ii) Let $u\in\Ran(\Gamma)\cap\Dom(\Pi)$. This implies that $u=(0,Su_0)$ for some $u_0 \in W^{1,2}(M)$ and
\begin{align*}
\|\Pi u\| = \|\Gamma^* u\| = \|S^*Su_0\|_{L^2(M)}
= \|u_0-\div\grad u_0\|_{L^2(M)}.
\end{align*}
It remains to prove that $\|u\|_{W^{1,2}(M;\C^{n+2})} \lesssim \|u_0-\div\grad u_0\|_{L^2(M)}$. Assuming  that $\iota_* \grad u_0$ is in $W^{1,2}(M;\C^n)$, we have
\begin{align*}
\|u\|&_{W^{1,2}(M;\C^{2+n})}^2 = \|Su_0\|_{W^{1,2}(M;\C^{1+n})}^2 \\
&= \|u_0\|_{W^{1,2}(M)}^2 + \|\iota_*\grad u_0\|_{W^{1,2}(M;\C^n)}^2 \\
&= \|u_0\|_{W^{1,2}(M)}^2 
+ \sum_{\alpha=1}^n \big(\|(\iota_* \grad u_0)_\alpha\|_{L^2(M)}^2
+ \|\nabla (\iota_* \grad u_0)_\alpha\|_{L^2(T^*M)}^2 \big) \\
&= \|u_0\|_{W^{1,2}(M)}^2 
+ \|\iota_*\grad u_0\|_{L^2(M;\C^n)}^2
+ \sum_{\alpha=1}^n \|\nabla (\grad u_0(\iota_\alpha))\|_{L^2(T^*M)}^2 \\
&= \|u_0\|_{L^2(M)}^2 + 2\|\nabla u_0\|_{L^2(T^*M)}^2 
+ \int \sum_{\alpha=1}^n |\nabla (\grad u_0(\iota_\alpha))|_{T_x^*M}^2\, \d\mu(x),
\end{align*}
where the final inequality follows from the fact that the Riemannian metric on $M$ is the metric induced by the embedding $\iota$. In particular, using \eqref{eq: induced.metric} we obtain
\[
|\iota_*\grad u_0|_{\C^n}^2 = |\grad u_0|_{T_xM}^2 = |\nabla u_0|_{T_x^*M}^2.
\]

To estimate $\sum_{\alpha=1}^n |\nabla (\grad u_0(\iota_\alpha))|_{T_x^*M}^2$, first consider $v\in C^\infty_c(M)$. In geodesic normal coordinates at a point $x\in M$, we obtain
\begin{align*}
&\sum_\alpha |\nabla (\grad v(\iota_\alpha))|^2_{T^*_x M}
= \partial_i(g^{kl}\partial_k v \partial_l \iota^\alpha) \partial_i(g^{mn}\partial_m v \partial_n \iota^\alpha) \\
&= g^{kl}g^{mn}\partial_i(\partial_k v \partial_l \iota^\alpha) \partial_i(
\partial_m v \partial_n \iota^\alpha) \\
&= \partial_i(\partial_k v \partial_k \iota^\alpha) \partial_i(
\partial_m v \partial_m \iota^\alpha) \\
&=  
(\partial_i\partial_k v \partial_k \iota^\alpha + \partial_k v \partial_i \partial_k \iota^\alpha)
(\partial_i\partial_m v \partial_m \iota^\alpha + \partial_m v \partial_i \partial_m \iota^\alpha) \\
&= (\partial_i\partial_k v \partial_k \iota^\alpha)(\partial_i\partial_m v \partial_m \iota^\alpha)
+(\partial_k v \partial_i \partial_k \iota^\alpha)(\partial_m v \partial_i \partial_m \iota^\alpha)
+2(\partial_m \iota^\alpha \partial_i \partial_k \iota^\alpha)(\partial_k v \partial_i\partial_m v).
\end{align*}
The decomposition $\C^n=T_xM\oplus N_xM$ is orthogonal, so by \eqref{eq: i*def} and \eqref{eq: II.loc.cord}, we have
\[
\partial_m \iota^\alpha \partial_i \partial_k \iota^\alpha
= \langle (\partial_m \iota^\alpha) e_\alpha, (\partial_i \partial_k \iota^\beta) e_\beta\rangle_{\C^n}
= \langle \iota_*\partial_m, h_{ik} \rangle_{\C^n}
=0.
\]
This shows that the last term in the previous estimate vanishes, so we have
\begin{align*}
\sum_\alpha |\nabla (\grad v(\iota_\alpha))|^2_{T^*_x M}
&= (\partial_k \iota^\alpha \partial_m \iota^\alpha)( \partial_i\partial_k v \partial_i\partial_m v) 
+(\partial_i \partial_k \iota^\alpha \partial_k v)( \partial_i \partial_m \iota^\alpha  \partial_m v) \\
&= g_{km}\partial_i\partial_k v \partial_i\partial_m v 
+\sum_{\alpha}\sum_{i}(\partial_i \partial_k \iota^\alpha \partial_k v)^2\\
&\leq \partial_i\partial_k v \partial_i\partial_k v 
+(\partial_i \partial_l \iota^\alpha\partial_i \partial_l \iota^\alpha)(\partial_k v\partial_k v) \\
&= \partial_i\partial_k v \partial_i\partial_k v + h_{il}^\alpha h_{il}^\alpha\partial_k v\partial_k v \\
&= |\nabla^2 v|_{T_x^*M \otimes T_x^*M}^2 + |h|_{T_x^*M\otimes T_x^*M\otimes N_xM}^2 |\nabla v|_{T_x^*M}^2,
\end{align*}
where we used \eqref{eq: metric} in the second line, the Cauchy--Schwartz inequality in the third, the local coordinate expression \eqref{eq: II.loc.cord} in the case of geodesic normal coordinates in the fourth, followed by \eqref{eq: nabla.function} and \eqref{eq: Hessian} in the fifth. Lemma~\ref{Lem: II.bdd.prop} shows that the Ricci curvature of $M$ is bounded below. Integrating the Bochner--Lichnerowicz--Weitzenb\"{o}ck formula and applying the Riemannian divergence theorem, as in Proposition~3.3 of \cite{Hebey1999}, then shows that
\begin{align*}
\|\nabla^2 v\|_{L^2(T^*M\otimes T^*M)}^2 
&\lesssim \|\nabla v\|_{L^2(T^*M)}^2 + \|\Delta v\|_{L^2(M)}^2,
\end{align*}
where the constant in the estimate depends only on the lower bound for the Ricci curvature. Altogether, the bound on the second fundamental form implies that
\begin{align} \begin{split}\label{eq: H8smooth}
\sum_{\alpha=1}^n \|\nabla (\grad v(\iota_\alpha))\|_{L^2(T^*M)}^2 
& \leq \|\nabla^2 v\|_{L^2(T^*M\otimes T^*M)}^2 + |h|^2 \|\nabla v\|_{L^2(T^*M)}^2\\
&\lesssim \|\nabla v\|_{L^2(T^*M)}^2 + \|\Delta v\|_{L^2(TM)}^2
\end{split}\end{align}
for all $v\in C_c^\infty(M)$.

We now use a density argument to show that the estimate above holds with $u_0$ in place of~$v$. To do this, let $\mathcal{K}^{2,2}(M)$ denote the space of all $u$ in $C^\infty(M)$ with
\[
\|u\|_{\mathcal{K}^{2,2}(M)}^2 := \|u\|_{L^2(M)}^2 + \|\nabla u\|_{L^2(T^*M)}^2 + \|\Delta u\|_{L^2(M)}^2 <\infty
\]
and define the Sobolev space $K^{2,2}(M)$ to be the completion of $\mathcal{K}^{2,2}(M)$ under the norm $\|\cdot\|_{\mathcal{K}^{2,2}(M)}$. This completion is identified with a subspace of $L^2(M)$ with norm $\|\cdot\|_{K^{2,2}}$ in the same way that $W^{1,2}(M)$ was identified in Section~\ref{Section: DiracTypeOperators}. Given that $u_0\in W^{1,2}(M)$ and $\div\grad u_0\in L^2(M)$, we can use Friedrichs' mollifiers to show that $u_0\in K^{2,2}(M)$ (see, for instance, Lemma~10.2.5 in \cite{HigsonRoe2000} or Appendix~A in \cite{BrMiSh2002}). It is shown in Propositions~3.2 and~3.3 of \cite{Hebey1999} that the space of smooth compactly supported functions is dense in $K^{2,2}$ on any complete Riemannian manifold with Ricci curvature bounded below and positive injectivity radius. Lemma~\ref{Lem: II.bdd.prop} shows that this is indeed the case on $M$. Therefore, since $u_0$ is in $K^{2,2}(M)$, choose a sequence $(u_n)_n$ in $C^\infty_c(M)$ that converges to $u_0$ in $K^{2,2}(M)$. Now introduce the notation $v_0^\alpha:=\grad u_0(\iota_\alpha)$ and $v_n^\alpha:=\grad u_n(\iota_\alpha)$. Using \eqref{eq: i*def} and \eqref{eq: induced.metric}, for all $n\in\N$ we have
\begin{align*}
\sum_{\alpha=1}^n \|v_n^\alpha-v_0^\alpha\|_{L^2(M)}^2 = \int |\iota_*\grad u_n-\iota_*\grad u_0|_{\C^n}^2 \ \d\mu
=\|\nabla u_n-\nabla u\|_{L^2(T^*M)}^2,
\end{align*}
and from \eqref{eq: H8smooth}, for all $m>n$ we have
\begin{align*}
\sum_{\alpha=1}^n \|\nabla (v_n^\alpha-v_m^\alpha)\|_{L^2(T^*M)}^2 \lesssim \|u_n-u_m\|_{K^{2,2}(M)}^2.
\end{align*}
This shows that $v_0^\alpha e_\alpha$ is in $W^{1,2}(M;\C^n)$ and from \eqref{eq: H8smooth}, for all $n\in\N$ we have
\begin{align*}
& \sum_{\alpha=1}^n \|\nabla (\grad u_0(\iota_\alpha))\|_{L^2(T^*M)}^2 
\leq \sum_{\alpha=1}^n \big(\|\nabla v_0^\alpha - \nabla v_n^\alpha\|_{L^2(T^*M)}^2 + \|\nabla v_n^\alpha\|_{L^2(T^*M)}^2\big)\\
&\lesssim \sum_{\alpha=1}^n \|\nabla v_0^\alpha - \nabla v_n^\alpha\|_{L^2(T^*M)}^2 + \|u_n-u_0\|_{K^{2,2}(M)}^2 + \|u_0\|_{K^{2,2}(M)}^2.
\end{align*}
This proves that $\iota_*\grad u_0$ is in $W^{1,2}(M;\C^n)$ and we conclude that 
\[
\|u\|_{W^{1,2}(M;\C^{2+n})}^2
\lesssim \|u_0\|_{L^2(M)}^2 + \|\nabla u_0\|_{L^2(T^*M)}^2 + \|\Delta u_0\|_{L^2(M)}^2,
\]
where the constant in the estimate depends only on $|h|$.

Finally, we use the definition of $\div$, the Cauchy--Schwartz inequality and the functional calculus for the self-adjoint operator $-\div\grad$ to obtain
\[
\|u\|_{W^{1,2}(M;\C^{2+n})}^2
\lesssim \|u_0\|_{L^2(M)}^2 + \|\Delta u_0\|_{L^2(M)}^2
\lesssim \|(I+\Delta)u_0\|_{L^2(M)}^2,
\]
which proves that $\|u\|_{W^{1,2}(M;\C^{2+n})} \lesssim \|u_0-\div\grad u_0\|_{L^2(M)} = \|\Pi u\|$.
\end{proof}

\section{Christ's Dyadic Cubes and Carleson Measures}\label{ChristCarleson}
The results in this section do not require a differentiable structure. To distinguish these results, it is convenient to let $\mathscr{X}$ denote a metric measure space with metric~$\rho$ and Radon measure $\mu$. A ball in $\mathscr{X}$ then refers to an open metric ball and the notation introduced above Definition~\ref{Def: Eloc} extends to this setting. The metric measure spaces we consider must at least satisfy the following local doubling condition.

\begin{Def}\label{DefLD} A metric measure space $\mathscr{X}$ is \textit{locally doubling} if for each $b>0$, there exists a constant $A_b\geq1$ such that
\begin{equation}\tag{D$_\text{loc}$}\label{LD}
0<V(x,2r)\leq A_b V(x,r)<\infty
\end{equation}
for all $x\in \mathscr{X}$ and $r\in(0,b]$.
\end{Def}

The proof of Theorem~\ref{Thm: mainquadraticestimate} in the case $M=\R^n$ in \cite{AKMc} relies on the dyadic cube structure of $\R^n$. In \cite{Christ}, Christ constructs a dyadic cube structure on a space of homogeneous type. This construction can be applied on a locally doubling metric measure space to provide a truncated dyadic cube structure. This is the content of the following proposition. The proof follows as in \cite{Christ}. In particular, the assumption that the measure $\mu$ is Radon instead of Borel is deliberate. This is because the proof of the thin boundary property (item (6) below) uses Lebesgue's differentiation theorem, which itself requires the density of continuous functions in $L^1$. This did not seem obvious for a Borel measure. In any case, the assumption serves our purposes because the Riemannian measure is Radon (see, for instance, Chapter~II.5 in \cite{Sakai1996}).

\begin{Prop} \label{Prop: dyadic.cubes}
Let $\mathscr{X}$ be a locally doubling metric measure space. There exists a countable collection $\Delta_{(0,1]}=(Q_\alpha^k)_{\alpha\in I_k, k\in\N_0}$ of open subsets of $\mathscr{X}$, indexed by some set $I_k$ for each integer $k\geq 0$, and a sequence $(x_\alpha^k)_{\alpha\in I_k,k\in\N_0}$ of points in $\mathscr{X}$, together with constants $\delta,\eta\in(0,1)$ and $a_0,a_1,c>0$, such that the following hold:
\begin{enumerate}
\item
$\mu\left(\mathscr{X}\setminus \bigcup_{\al\in I_k} Q_\al^k\right)=0$;
\item
$Q_\al^k \cap Q_\beta^k = \emptyset$ for all $\al,\beta\in I_k$;
\item For each $l>k$, $\al\in I_k$ and $\beta\in I_l$, either $Q_\beta^l \subset Q_\al^k$ or
$Q_\al^k \cap Q_\beta^l = \emptyset$;
\item
For each $l \in [0,k)$ and $\al\in I_k$  there is a unique $\beta\in I_l$ such
that $Q_\al^k \subset Q_\beta^l$;
\item
$B(x_\al^k, a_0\delta^k) \subseteq Q_\al^k \subseteq B(x_\al^k, a_1\delta^k)$
for all $\al\in I_k$;
\item
$\mu(\{x\in Q_\alpha^k : \rho(x,\mathscr{X} \setminus Q_\al^k) \leq s\delta^k\}) \leq cs^\eta \mu(Q_\al^k)$ for all $\alpha\in I_k$ and $s>0$.
\end{enumerate}
\end{Prop}

Any collection of sets $\Delta_{(0,1]}=(Q_\alpha^k)_{\alpha\in I_k, k\in\N_0}$ with the properties in Proposition~\ref{Prop: dyadic.cubes} is called a \textit{truncated dyadic cube structure on $\mathscr{X}$}; the sets in $\Delta_{(0,1]}$ are called dyadic cubes. Given $t\in(0,1]$, define the collection of dyadic cubes $\Delta_t:= (Q_\alpha^{k})_{\alpha\in I_{k}}$ by requiring that $k\in\N_0$ satisfy $\delta^{k+1} < t \leq \delta^k$. For all $Q\in\Delta_t$, define the side length of $Q$ by $l(Q):=\delta^k$ and the Carleson box over $Q$ by $C(Q):=Q\times(0,l(Q)]$. Note that $t\leq l(Q) < t/\delta$ and so $l(Q)\eqsim t$. The dyadic averaging operator $A_t$ is then defined for all $u\in L^1_{\text{loc}}(\mathscr{X})$ by
\[
A_tu(x)=\dashint_Q u(y)\ \d\mu(y) := \frac{1}{\mu(Q)} \int_Q u(y)\ \d\mu(y)
\]
for all $t\in(0,1]$ and almost all $x\in \mathscr{X}$, where $Q$ is the unique dyadic cube in $\Delta_t$ containing $x$. Standard arguments, such as those in Section~2 of~\cite{CMcM}, show that local doubling is enough to guarantee that the following dyadic maximal operator
\[
\mathcal{M}_{\Delta_{(0,1]}}u(x) := \sup_{t\in(0,1]} A_tu(x)\quad\text{for all}\quad x\in \mathscr{X}
\]
is bounded on $L^p(\mathscr{X})$ for all all $p\in(1,\infty]$.

Given a truncated dyadic cube structure on $\mathscr{X}$ with the constants specified in Proposition~\ref{Prop: dyadic.cubes}, the constant $a:=\max\{1, a_1/\delta\}$. It is useful to record the following inequalities, which will be used frequently. Given $t\in(0,1]$, dyadic cubes $Q,R\in\Delta_t$ and points $x_Q,x_R\in \mathscr{X}$ such that \[B(x_Q,a_0l(Q)) \subseteq Q \subseteq B(x_Q,a_1l(Q)) \quad\text{and}\quad B(x_R,a_0l(R)) \subseteq R \subseteq B(x_R,a_1l(R)),\] the following are easily verified:
\begin{align}\begin{split}\label{eq: xQ.x.QR}
\rho(Q,R) &\leq \rho(x_Q,x_R) \leq a\left(2t+\rho(Q,R)\right); \\
\rho(Q,R) &\leq \ \rho(Q,x) \ \ \leq a\left(2t+\rho(Q,R)\right)\quad\text{for all}\quad x\in R; \\
\rho(Q,x) &\leq \ \rho(x_Q,x) \ \leq \ a\left(t+\rho(Q,x)\right)\quad \ \text{for all}\quad x\in \mathscr{X}.
\end{split}\end{align}

In the next section, we reduce the proof of Theorem~\ref{Thm: mainquadraticestimate} to verifying a local analogue of Carleson's condition. This relies on the following result. The approach taken in the proof was suggested by an examiner of the author's thesis.

\begin{Thm}\label{thm: local.carleson.tentspaceequiv}
Let $\mathscr{X}$ be a locally doubling metric measure space and suppose that $t_0\in(0,1]$. If $\nu$ is a (positive) measure on $\mathscr{X}\times(0,t_0]$ satisfying the following local analogue of Carleson's condition
\[
\|\nu\|_{\mathcal{C}}:=\sup_{t\in(0,t_0]}\sup_{Q\in\Delta_t} \frac{1}{\mu(Q)}\iint_{C(Q)}\ d\nu(x,t) < \infty,
\]
then
\[
\iint_{\mathscr{X}\times(0,t_0]} |A_tu(x)|^2\ \d\nu(x,t)
\lesssim \|\nu\|_{\mathcal{C}}\|u\|^2
\]
for all $u\in L^2(\mathscr{X})$.
\end{Thm}
\begin{proof}
Fix $K$ so that $\delta^{K+1} < t_0 \leq \delta^K$. Using the notation for dyadic cubes from above, we have
\begin{align*}
\iint_{\mathscr{X}\times(0,t_0]} |A_tu(x)|^2\ \d\nu(x,t)
&\leq \sum_{k=K}^\infty \sum_{\alpha\in I_k} \int_{\delta^{(k+1)}}^{\delta^{k}} \int_{Q^k_\alpha} \left| \dashint_{Q^k_\alpha} u(y)\ \d\mu(y)\right|^2 \d\nu(x,t).
\end{align*}
Introducing the notation
\[
u_{\alpha,k} = \dashint_{Q^k_\alpha} u(y)\ \d\mu(y)
\quad\text{and}\quad
\nu_{\alpha,k} = \nu\left(
Q^k_\alpha \times (\delta^{(k+1)},\delta^{k}]\right),
\]
Fubini's Theorem shows that the preceding expression on the right is equal to
\begin{align*}
\sum_{k=K}^\infty \sum_{\alpha\in I_k} |u_{\alpha,k}|^2 \nu_{\alpha,k} 
&= \sum_{k=K}^\infty \sum_{\alpha\in I_k}  \nu_{\alpha,k} \int_0^{|u_{\alpha,k}|} 2r\ \d r\\
&= \int_0^\infty 2r \sum_{k=K}^\infty \sum_{\alpha\in I_k} \ca_{\{|u_{\alpha,k}|>r\}} \nu_{\alpha,k} \ \d r,
\end{align*}
where $\d r$ denotes Lebesgue measure on $(0,\infty)$.

For each $r>0$, let $(R_{j}(r))_{j\in\N}$ denote an enumeration of the collection of maximal dyadic cubes $Q_\alpha^k$ in $\Delta_{(0,1]}$ such that $|u_{\alpha,k}|>r$. This collection has the property that
\[
\bigcup_{j=1}^\infty R_j(r) = \{x\in \mathscr{X}\ |\ \mathcal{M}_{\Delta_{(0,1]}}u(x) > r \}
\]
and so
\begin{align*}
\sum_{k=K}^\infty \sum_{\alpha\in I_k} |u_{\alpha,k}|^2 \nu_{\alpha,k}
& \leq \int_0^\infty  2r  \sum_{j=1}^\infty \sum_{\substack{R\in\Delta_{(0,1]};\\ R\subseteq R_j(r)}} \nu\big(
R \times (\delta l(R),l(R)]\big)\, \d r \\
&= \int_0^\infty 2r \sum_{j=1}^\infty \nu\big(
C(R_j(r))\big)\, \d r\\
&\leq \|\nu\|_{\mathcal{C}}\int_0^\infty 2r\, \mu(\{x\in \mathscr{X}\ |\ \mathcal{M}_{\Delta_{(0,1]}}u(x) > r \})\, \d r \\
&= \|\nu\|_{\mathcal{C}} \|\mathcal{M}_{\Delta_{(0,1]}} u \|^2\\
&\lesssim \|\nu\|_{\mathcal{C}} \|u\|^2.
\end{align*}
\end{proof}

We conclude this section by recording two technical results for use later on. For all $x\geq0$, we use the notation $\langle x\rangle:=\min\{1,x\}$.

\begin{Lem}\label{Lem: muQmuR}
Let $\mathscr{X}$ be a metric measure space satisfying \eqref{E}. Let $\Delta_{(0,1]}$ denote a truncated unit cube structure on $\mathscr{X}$ with the constant $a:=\max\{1,a_1/\delta\}$. Then
\[
\left\langle\frac{t}{\rho(Q,R)}\right\rangle^{\kappa} e^{-a\lambda \rho(Q,R)} \lesssim \frac{\mu(Q)}{\mu(R)} \lesssim \left\langle\frac{t}{\rho(Q,R)}\right\rangle^{-\kappa} e^{a\lambda \rho(Q,R)}
\]
for all $Q,R\in\Delta_t$ and $t\in(0,1]$.
\end{Lem}

\begin{proof}
It suffices to show the second inequality, since the estimate is symmetric in $R$ and $Q$. It follows from \eqref{E} that
\[
V(x,r) \leq A\left(1+\frac{\rho(x,y)}{r}\right)^\kappa e^{\lambda (r+\rho(x,y))} V(y,r),
\]
for all $x,y\in \mathscr{X}$ and $r>0$, since $B(x,r)\subseteq B(y,(1+\frac{\rho(x,y)}{r})r)$. Given $t\in(0,1]$ and $Q,R\in\Delta_t$, it then follows from Proposition~\ref{Prop: dyadic.cubes} and \eqref{E} that there exists $x_Q,x_R\in \mathscr{X}$ such that
\begin{align*}
\frac{\mu(Q)}{\mu(R)} &\leq \frac{V(x_Q,a_1l(Q))}{V(x_R,a_0l(R))} \\
&\leq A (\tfrac{a_1}{a_0})^\kappa e^{\lambda a_1l(Q)} \frac{V(x_Q,a_0l(Q))}{V(x_R,a_0l(R))} \\
&\lesssim A\left(1+\frac{\rho(x_Q,x_R)}{a_0l(Q)}\right)^\kappa e^{\lambda (a_0l(Q)+\rho(x_Q,x_R))} \\
&\lesssim \left(1+\frac{\rho(x_Q,x_R)}{t}\right)^\kappa e^{\lambda \rho(x_Q,x_R)}.
\end{align*}
For all $x>0$, we have $1+x \leq 2\max\{1,x\} = 2\langle 1/x \rangle^{-1}$. Using this and the above estimate with \eqref{eq: xQ.x.QR}, we conclude that
\[
\frac{\mu(Q)}{\mu(R)}
\lesssim \left\langle\frac{t}{\rho(x_Q,x_R)}\right\rangle^{-\kappa} e^{\lambda \rho(x_Q,x_R)}
\lesssim \left\langle\frac{t}{\rho(Q,R)}\right\rangle^{-\kappa} e^{a\lambda \rho(Q,R)}
\]
for all $Q,R\in \Delta_t$.
\end{proof}

\begin{Lem}\label{Lem: CubeSums}
Let $\mathscr{X}$ be a metric measure space satisfying \eqref{E}. Let $\Delta_{(0,1]}$ denote a truncated unit cube structure on $\mathscr{X}$. If $t\in(0,1]$, $M>\kappa$ and $m>\lambda t$, then
\[
\sup_{R\in\Delta_t} \sum_{Q\in\Delta_t}  \frac{\mu(Q)}{\mu(R)} \left\langle\frac{t}{\rho(Q,R)}\right\rangle^{M} e^{-m \frac{\rho(Q,R)}{t}} \lesssim 1.
\]
\end{Lem}

\begin{proof}
Suppose that $t\in(0,1]$, $M>\kappa$ and $m>\lambda t$. Let $\sigma=m/\lambda t >1$ and for each $R\in\Delta_t$, let
\[
\Delta_t^j(R) =
\begin{cases}
\{Q\in \Delta_t : \rho(Q,R)/t \leq 1\}  &{\rm if}\quad j=0; \\
\{Q\in \Delta_t : \sigma^{j-1}<\rho(Q,R)/t \leq \sigma^j\}  &{\rm if}\quad j\in\N.
\end{cases}
\]
For each $R\in\Delta_t$, Proposition~\ref{Prop: dyadic.cubes} implies that there exists $x_R\in \mathscr{X}$ such that
\[
B(x_R,a_0l(R)) \subseteq R \subseteq B(x_R,a_1l(R)).
\]
A simple calculation then shows that
\[
\bigcup \Delta_t^j(R) \subseteq B(x_R,3a_1l(R)+\sigma^jt)
\]
for all $j\in\N_0$, and it follows from \eqref{E} that
\[
\mu\big(\bigcup \Delta_t^j(R)\big) \lesssim \sigma^{j\kappa}e^{\lambda \sigma^jt} \mu(R)
\]
for all $j\in\N_0$. Therefore, we have
\begin{align*}
\sup_{R\in\Delta_t} \sum_{Q\in\Delta_t} & \frac{\mu(Q)}{\mu(R)} \left\langle\frac{t}{\rho(Q,R)}\right\rangle^{M} e^{-m \frac{\rho(Q,R)}{t}}
\\
&=\sup_{R\in\Delta_t} \sum_{j=0}^\infty \sum_{Q\in\Delta_t^j(R)}  \frac{\mu(Q)}{\mu(R)} \left\langle\frac{t}{\rho(Q,R)}\right\rangle^{M} e^{-m \frac{\rho(Q,R)}{t}} \\
&\leq \sup_{R\in\Delta_t} \frac{1}{\mu(R)} \bigg[\mu\big(\bigcup\Delta_t^0(R)\big)+
\sum_{j=1}^\infty \sigma^{-(j-1)M}e^{-m\sigma^{j-1}} \mu\big(\bigcup \Delta_t^j(R)\big) \bigg]
\\
&\lesssim \sum_{j=0}^\infty \sigma^{-j(M-\kappa)}e^{-(m-\sigma\lambda t)2^{j-1}}
\\
&\lesssim 1,
\end{align*}
as required.
\end{proof}

\section{The Main Local Quadratic Estimate}\label{Section: MainLQE}
This section contains the proof of Theorem~\ref{Thm: mainquadraticestimate}. We consider a complete Riemannian manifold $M$ satisfying \eqref{E} and \eqref{P} with constants $\kappa,\lambda\geq 0$, and suppose that $\{\Gamma,B_1,B_2\}$ are operators on $L^2(M;\C^N)$ satisfying the hypotheses (H1)--(H8) from Section~\ref{Section: DiracTypeOperators}. Let $\|\cdot\|$ and $\langle\cdot,\cdot\rangle$ denote the norm and inner-product on $L^2(M;\C^N)$. Let $|\cdot|$ and $(\cdot,\cdot)$ denote the norm and inner-product on $\C^N$. Fix a truncated dyadic cube structure $\Delta_{(0,1]}$ with constants $\delta,\eta\in(0,1)$ and $a_1>a_0>0$ as in Proposition~\ref{Prop: dyadic.cubes} and set $a:=\max\{1,a_1/\delta\}$. For all $x\geq0$, there is the notation $\langle x\rangle:=\min\{1,x\}$. We follow \cite{AKMc,AKMc2} and introduce the following operators.

\begin{Def}
Given $t\in\R\setminus\{0\}$, define the following bounded operators:
\begin{align*}
R_t^B&:=(I+it\Pi_B)^{-1};\\
P_t^B&:= (I+t^2\Pi_B^2)^{-1} = \tfrac{1}{2}(R_t^B+R_{-t}^B); \\
Q_t^B&:= t\Pi_B(I+t^2\Pi_B^2)^{-1} = \tfrac{1}{2i}(-R_t^B+R_{-t}^B) ; \\
\Theta_t^B&:= t\Gamma_B^*(I+t^2\Pi_B^2)^{-1}.
\end{align*}
The operators $R_t$, $P_t$ and $Q_t$ are defined as above by replacing $\Pi_B$ with $\Pi$.
\end{Def}

The uniform estimate
\begin{equation}\label{eq: UniformEstRQP}
\sup_{t\in\R\setminus\{0\}} \|U_t\| \lesssim 1
\end{equation}
holds when $U_t=R_t^B$, $P_t^B$, $Q_t^B$ and $\Theta_t^B$. This follows immediately from \eqref{eq: GGB} and the resolvent bounds in \eqref{eq: KatoSw}, since $R_t^B = ({i}/{t}) [({i}/{t})I-\Pi_B]^{-1}$ for all $t\in\R\setminus\{0\}$.

The operator $\Pi$ is self-adjoint, so by the functional calculus for self-adjoint operators, we have the quadratic estimate
\begin{equation}\label{eq: unper.quad.est}
\int_0^\infty \|Q_tu\|^2\ \frac{\d t}{t} \eqsim \|u\|^2
\end{equation}
for all $u\in \overline{\Ran(\Pi)}$.

The following result, which is an immediate consequence of Proposition~4.8 in~\cite{AKMc} and the inhomogeneity assumed in hypothesis (H8), shows that Theorem~\ref{Thm: mainquadraticestimate} can be reduced to finding $t_0>0$ small enough such that a certain local quadratic estimate holds. The theory of local quadratic estimates was developed by the author in~\cite{Morris2010(1)}.

\begin{Prop}\label{Prop: t0reduction}
If there exists $t_0\in(0,1]$ such that
\begin{equation}\label{eq: red.est}
\int_0^{t_0} \|\Theta_t^B P_tu\|^2\ \frac{\d t}{t} \lesssim \|u\|^2
\end{equation}
for all $u\in\Ran(\Gamma)$, as well as the three similar estimates obtained upon replacing $\{\Gamma,B_1,B_2\}$ by $\{\Gamma^*,B_2,B_1\}$, $\{\Gamma^*,B_2^*,B_1^*\}$ and $\{\Gamma,B_1^*,B_2^*\}$, then \eqref{eq: mainquadraticestimate} holds for all $u$ in $\overline{\Ran(\Pi_B)}$.
\end{Prop}
\begin{proof}
Suppose that there exists $t_0\in(0,1]$ such that \eqref{eq: red.est} holds for all $u\in\Ran(\Gamma)$, as well as the three similar estimates mentioned in the proposition. If $u\in \Ran(\Gamma)$ and $t>0$, then $P_t u = u-t\Pi Q_tu \in \Ran(\Pi)$, since the Hodge decomposition guarantees that $\Ran(\Gamma)\subseteq\Ran(\Pi)$. Therefore, hypothesis (H8) implies that $\|P_tu\| \lesssim \|\Pi P_t u\|$ for all $u\in\Ran(\Gamma)$ and $t>0$. The uniform bound in \eqref{eq: UniformEstRQP} then implies that
\[
\int_{t_0}^\infty \|\Theta_t^B P_tu\|^2\ \frac{\d t}{t}
\lesssim \int_{t_0}^\infty \|Q_t u\|^2\ \frac{\d t}{t^3}
\lesssim \|u\|^2 \int_{t_0}^\infty \frac{\d t}{t^3}
\lesssim \|u\|^2
\]
for all $u\in\Ran(\Gamma)$, which shows that
\[
\int_{0}^\infty \|\Theta_t^B P_tu\|^2\ \frac{\d t}{t} \lesssim \|u\|^2
\]
for all $u\in\Ran(\Gamma)$, as well as the three similar estimates obtained upon replacing $\{\Gamma,B_1,B_2\}$ by $\{\Gamma^*,B_2,B_1\}$, $\{\Gamma^*,B_2^*,B_1^*\}$ and $\{\Gamma,B_1^*,B_2^*\}$. It then follows from Proposition~4.8 in~\cite{AKMc} that \eqref{eq: mainquadraticestimate} holds for all $u$ in $\overline{\Ran(\Pi_B)}$.
\end{proof}

The above result allows us to work locally, in the sense that we only need to consider $t\in(0,1]$, which means that we are not restricted to considering manifolds that are doubling. The metric-measure interaction is instead restricted by the exponential nature of the following off-diagonal estimates. We follow the proof of Proposition~5.2 in \cite{AKMc} by Axelsson, Keith and McIntosh, and apply the exponential off-diagonal estimates from \cite{CMcM} by Carbonaro, McIntosh and the author.

\begin{Prop}\label{Prop: ODestKato} Let $U_t$ denote either $R_t^B$, $P_t^B$, $Q_t^B$ or $\Theta_t^B$ for all $t\in\R\setminus\{0\}$. There exists a constant $C_\Theta>0$, which depends only on the constants in (H1)--(H8), such that the following holds: For each $M\geq0$, there exists $c>0$ such that
\[
\|\ca_EU_t\ca_F\| \leq c \left\langle\frac{|t|}{\rho(E,F)}\right\rangle^M \exp\left(-C_\Theta \frac{\rho(E,F)}{|t|}\right)
\]
for all closed subsets $E$ and $F$ of $M$.
\end{Prop}
\begin{proof}
In the case $U_t=R_t^B=({i}/{t}) [({i}/{t})I-\Pi_B]^{-1}$, the result follows exactly as in the proof of Lemma~5.3 in \cite{CMcM}, since $\Pi_B$ is of type $S_{\omega}$ and (H5)--(H6) imply that
\begin{align*}
|[\Pi_B,\eta I]u(x)|
&= |[\Gamma,\eta I]u(x)+B_1[\Gamma^*,\eta I]B_2u(x)|\\
&\leq C_\Gamma(1+\|B_1\|_\infty \|B_2\|_\infty) |\nabla \eta(x)|_{T_xM}|u(x)|
\end{align*}
for all $\eta\in C_c^\infty(M)$, $u\in\Dom(\Pi_B)$ and almost all $x\in M$. The results for $P_t^B$ and $Q_t^B$ then follow by linearity. In the case $U_t=Q_t^B$, the result is also given by the proof of Lemma~5.4 in \cite{CMcM}.

Now consider $U_t=\Theta_t^B=t\Gamma^*_BP_t^B$. Let $E$ and $F$ be closed subsets of $M$ with $\rho(E,F)>0$. Let $\tilde E = \{x\in M: \rho(x,E)\leq \rho(E,F)/2\}$ and choose $\eta:M\rightarrow[0,1]$ in $C^\infty_c(M)$ supported on $\tilde {E}$, equal to $1$ on $E$ and satisfying
$\|\nabla \eta\|_{\infty} \leq 3/{\rho(E,F)}$. The function $\eta$ can be constructed from smooth approximations of the Lipschitz function $\widetilde\eta$ that is supported on $\tilde{E}$ and defined by $\widetilde\eta(x)=1-2\rho(x,E)/\rho(E,F)$ for all $x\in\tilde{E}$; this is Lipschitz because the geodesic distance is Lipschitz on a Riemannian manifold. For further details see, for instance, \cite{AzagraFerreraLopezRangel2007}. Using both \eqref{eq: GGB} and (H5)--(H6), we obtain
\begin{align*}
\|\ca_E\Theta_t^B\ca_F\| &\leq \|(\eta I)t\Gamma^*_BP_t^B\ca_F\|\\
&\leq \|t[\eta I,\Gamma^*_B]P_t^B\ca_F\|+\|t\Gamma^*_B(\eta I)P_t^B\ca_F\|\\
&\lesssim |t|\|\nabla \eta\|_{\infty} \|\ca_{\tilde E}P_t^B\ca_F\|+\|t\Pi_B(\eta I)P_t^B\ca_F\|\\
&\leq |t|\|\nabla \eta\|_{\infty} \|\ca_{\tilde E}P_t^B\ca_F\|+|t|\|[\Pi_B,(\eta I)]P_t^B\ca_F\| +\|(\eta I)Q_t^B\ca_F\|\\
&\leq |t|\|\nabla \eta\|_{\infty} \|\ca_{\tilde E}P_t^B\ca_F\|+|t|\|\nabla \eta\|_{\infty}\|\ca_{\tilde E}P_t^B\ca_F\| +\|\ca_{\tilde E}Q_t^B\ca_F\|
\end{align*}
for all $t\in\R\setminus\{0\}$. The result then follows from the corresponding estimates for $P_t^B$ and $Q_t^B$, since $\rho(\tilde E,F) = 2 \rho(E,F)$.
\end{proof}

The off-diagonal estimates imply the following result.
\begin{Lem}\label{Lem: Theta.ext}
The operator $\Theta^B_t$ on $L^2(M;\C^N)$ has a bounded extension
\[
\Theta^B_t: L^\infty(M;\C^N) \rightarrow L^2_{\text{loc}}(M;\C^N)
\]
for all $t\in(0,\langle C_\Theta/2a\lambda\rangle]$. Moreover, there exists $c>0$ such that
\[
\|\Theta^B_tu\|_{L^2(Q)}^2 \leq c \mu(Q) \|u\|_\infty^2
\]
for all $u\in L^\infty(M;\C^N)$, $Q\in\Delta_t$ and $t\in(0,\langle C_\Theta/2a\lambda\rangle]$.
\end{Lem}

\begin{proof}
Let $t\in(0,\langle C_\Theta/2a\lambda\rangle]$ and $Q\in\Delta_t$. There exists $x_Q\in M$ such that
\[
B(x_Q,a_0t)\subseteq Q \subseteq B(x_Q,(a_1/\delta)t).
\]
Let $\Delta_t^{m,n}(Q) = \{R\in \Delta_t : m<\rho(Q,R) \leq n\}$ for all integers $n>m\geq 0$. Let $u\in L^\infty(M;\C^N)$ and define $u_n=\ca_{\Delta_t^{0,n}(Q)}u$ for all $n\in\N$. If $n>m$, then
\begin{align*}
\|&\Theta_t^B(u_n-u_m)\|_{L^2(Q)}^2 \\
&\leq \bigg[\sum_{R\in\Delta_t^{m,n}(Q)}\left(\frac{\mu(R)}{\mu(Q)} \frac{\mu(Q)}{\mu(R)}\right)^{\frac{1}{2}}\|\ca_Q\Theta_t^B\ca_R\| \|\ca_Ru\| \bigg]^2
\\
&\leq \bigg(\sum_{R\in\Delta_t^{m,n}(Q)} \frac{\mu(R)}{\mu(Q)}\|\ca_Q\Theta_t^B\ca_R\|\bigg) \sum_{R\in\Delta_t^{m,n}(Q)} \frac{\mu(Q)}{\mu(R)}\|\ca_Q\Theta_t^B\ca_R\| \|\ca_Ru\|^2
\\
&\leq \bigg(\sum_{R\in\Delta_t^{m,n}(Q)} \frac{\mu(R)}{\mu(Q)}\|\ca_Q\Theta_t^B\ca_R\|\bigg) \sum_{R\in\Delta_t^{m,n}(Q)} \frac{\mu(Q)}{\mu(R)}\|\ca_Q\Theta_t^B\ca_R\| \|u\|_\infty^2 \mu(R) \\
&\leq \bigg(\sum_{R\in\Delta_t^{m,n}(Q)} \frac{\mu(R)}{\mu(Q)}\|\ca_Q\Theta_t^B\ca_R\|\bigg) \bigg(\sum_{R\in\Delta_t^{m,n}(Q)} \frac{\mu(R)}{\mu(Q)} \frac{\mu(Q)}{\mu(R)} \|\ca_Q\Theta_t^B\ca_R\|\bigg)\mu(Q)\|u\|_\infty^2,
\end{align*}

Now choose $M>2\kappa$. The off-diagonal estimates from Proposition~\ref{Prop: ODestKato} then show that
\[
\|\ca_Q\Theta^B_t\ca_R\| \lesssim \left\langle\frac{t}{\rho(Q,R)}\right\rangle^M \exp\left(-C_\Theta \frac{\rho(Q,R)}{t}\right)
\]
for all $Q,R\in\Delta_t$ and $t\in(0,1]$. Moreover, Lemma~\ref{Lem: muQmuR} shows that
\[
\frac{\mu(Q)}{\mu(R)} \lesssim \left\langle\frac{t}{\rho(Q,R)}\right\rangle^{-\kappa} e^{a\lambda \rho(Q,R)}
\]
for $Q,R\in\Delta_t$ and $t\in(0,1]$. Then, since $M-\kappa>\kappa$ and $C_\Theta-a\lambda t > \lambda t$, Lemma~\ref{Lem: CubeSums} guarantees that both of the partial sums in the estimate above converge. Therefore, the sequence $(\Theta_t^Bu_n)_n$ is Cauchy in $L^2(Q)$ and
\[
\sup_{n\in\N} \|\Theta_t^Bu_n\|_{L^2(Q)}^2 \lesssim \mu(Q) \|u\|_\infty^2
\]
for all $Q\in\Delta_t$ and $t\in(0,\langle C_\Theta/2a\lambda\rangle]$, which implies the result.
\end{proof}

As in \cite{AKMc,AKMc2}, we now introduce the following operator to prove~\eqref{eq: red.est}.

\begin{Def}
For each $w\in \C^N$, let $\tilde w\in L^\infty(M;\C^N)$ denote the constant function that is equal to $w$ on $M$. For each $x\in M$ and $t\in(0,\langle C_\Theta/\lambda\rangle]$, the multiplication operator $\gamma_t(x) \in \mathcal{L}(\C^N)$ is defined by
\[
[\gamma_t(x)]w:=(\Theta^B_t \tilde w)(x)
\]
for all $w\in\C^N$, where $\Theta^B_t$ is defined on $L^\infty(M;\C^N)$ by Lemma~\ref{Lem: Theta.ext}.
\end{Def}

\begin{Cor}
The functions $\gamma_t:=(x\mapsto \gamma_t(x)\ \forall\ x\in M$) are in $L^2_{\text{loc}}(M;\mathcal{L}(\C^N))$ and there exists $c>0$ such that
\begin{equation}\label{eq: gamma.est}
\dashint_Q |\gamma_t(x)|^2\ \d\mu(x) \leq c
\end{equation}
for all $Q\in\Delta_t$ and $t\in(0,\langle C_\Theta/2a\lambda\rangle]$. Moreover, $\sup_{t\in(0,\langle C_\Theta /2a\lambda\rangle ]}\|\gamma_tA_t\|\lesssim1$.
\end{Cor}

\begin{proof}
The first property follows from Proposition~\ref{Lem: Theta.ext} and the definition of $\gamma_t$. It then follows that
\begin{align*}
\|\gamma_tA_tu\|^2 &= \sum_{Q\in\Delta_t} \int_Q |\gamma_t(y)A_tu(y)|^2\ \d\mu(y) \\
&= \sum_{Q\in\Delta_t} \int_Q \left|\gamma_t(y)\dashint_Q u(x)\ \d\mu(x)\right|^2\ \d\mu(y) \\
&= \sum_{Q\in\Delta_t} \left|\dashint_Q u(x)\ \d\mu(x)\right|^2 \int_Q |\gamma_t(y)|^2\ \d\mu(y) \\
&\lesssim \sum_{Q\in\Delta_t} \|u\|_{L^2(Q)}^2\\
&= \|u\|^2
\end{align*}
for all $t\in(0,\langle C_\Theta /2a\lambda\rangle]$ and $u\in L^2(M;\C^N)$, which completes the proof.
\end{proof}

To prove \eqref{eq: red.est}, we follow \cite{AKMc,AKMc2} and estimate each of the following terms separately:
\begin{align}\begin{split}\label{eq: red.est.3parts}
\int_0^{t_0} \|\Theta_t^B P_tu\|^2\ \frac{\d t}{t}
&\lesssim \int_0^{t_0} \|\Theta_t^B P_tu-\gamma_tA_tP_tu\|^2\ \frac{\d t}{t}
+\int_0^{t_0} \|\gamma_tA_t(P_t-I)u\|^2\ \frac{\d t}{t}\\
&\quad+\int_0^{t_0} \int_M|A_tu(x)|^2|\gamma_t(x)|^2\ \frac{\d\mu(x)\d t}{t}.
\end{split}\end{align}

The following weighted Poincar\'{e} inequality is used to estimate the first term above. The proof is based on techniques contained in Lemma~5.4 of \cite{AKMc} that have been adapted to suit off-diagonal estimates of exponential type.

\begin{Lem}\label{Lem: Poincare}
Given $M>\kappa+3$ and $m\geq a\lambda$, we have
\begin{align*}
\int_M |u(x)&-u_Q|^2 \left\langle\frac{t}{\rho(x,Q)}\right\rangle^M e^{-m \rho(x,Q)/t}\ \d\mu(x)\\
&\lesssim t^2 \int_M (|u(x)|^2+|\nabla u(x)|_{T^*_xM}^2)\left\langle \frac{t}{\rho(x,Q)}\right\rangle^{M-(\kappa+3)} e^{-(\frac{m}{a}-\lambda t)  \rho(x,Q)/t}\ \d\mu(x)
\end{align*}
for all $u\in W^{1,2}(M)$, $Q\in\Delta_t$ and $t\in(0,1]$.
\end{Lem}

\begin{proof}
Let $t\in(0,1]$ and $Q\in\Delta_t$. There exists $x_Q\in M$ such that
\[
B(x_Q,a_0t)\subseteq Q \subseteq B(x_Q,(a_1/\delta)t).
\]
Let $r\geq a$ and $u\in W^{1,2}(M)$. We have
\[
\|\ca_{B(x_Q,rt)}(u-u_Q)\|_2^2
\leq \|\ca_{B(x_Q,rt)}(u-u_{B(x_Q,rt)})\|_2^2 + \|\ca_{B(x_Q,rt)}(u_{B(x_Q,rt)}-u_Q)\|_2^2.
\]
The Cauchy--Schwartz inequality and \eqref{E} imply that
\begin{align*}
\|\ca_{B(x_Q,rt)}(u_{B(x_Q,rt)}-u_Q)\|_2^2 &= V(x_Q,rt)|u_Q-u_{B(x_Q,rt)}|^2 \\
&= V(x_Q,rt)\left|\dashint_Q (u-u_{B(x_Q,rt)}) \right|^2 \\
&\leq \frac{V(x_Q,rt)}{\mu(Q)}\int_Q |u-u_{B(x_Q,rt)}|^2\\
&\lesssim r^\kappa e^{\lambda rt}  \|\ca_{B(x_Q,rt)}(u-u_{B(x_Q,rt)})\|_2^2,
\end{align*}
where $r\geq a_1 /\delta$ ensured that $Q\subseteq B(x_Q,rt)$. It then follows from \eqref{P} that
\begin{align*}
\|\ca_{B(x_Q,rt)}&(u-u_Q)\|_2^2
\lesssim (1+r^\kappa e^{\lambda r t}) (rt)^2 (\|\ca_{B(x_Q,rt)}u\|_2^2+\|\ca_{B(x_Q,rt)}\nabla u\|_{L^2(T^*M)}^2).
\end{align*}

Now let $\nu(r):=-r^{-M}e^{-(m/a) r}$ for all $r\geq a$, in which case
\[
\d\nu(r)= (M r^{-M-1}+(m/a) r^{-M})e^{-(m/a) r}\d r
 \]
is a positive measure on $(a,\infty)$. Integrating the above estimate with respect to $\nu$, we obtain
\begin{align*}
\int_{a}^\infty\int_M\ca_{B(x_Q,rt)} &|u(x)-u_Q|^2\ \d\mu(x)\d\nu(r)\\
&\lesssim t^2 \int_{a}^\infty r^{\kappa+2} e^{\lambda rt} \int_M\ca_{B(x_Q,rt)} (|u(x)|^2+|\nabla u(x)|_{T^*_xM}^2)\ \d\mu(x)\d\nu(r).
\end{align*}
It then follows from \eqref{eq: xQ.x.QR} and Fubini's theorem that
\begin{align*}
\int_M&|u(x)-u_Q|^2 \left\langle\frac{t}{\rho(x,Q)}\right\rangle^{M}e^{-m \rho(x,Q)/t}\ \d\mu(x)\\
&\lesssim \int_M|u(x)-u_Q|^2 \left\langle\frac{t}{\rho(x,x_Q)}\right\rangle^{M}e^{-\frac{m}{a} \rho(x,x_Q)/t}\ \d\mu(x)\\
&\lesssim \int_M|u(x)-u_Q|^2 (\max\{\rho(x,x_Q)/t,a\})^{-M}e^{-\frac{m}{a} \max\{\rho(x,x_Q)/t,a\}}\ \d\mu(x)\\
&= \int_M|u(x)-u_Q|^2 \int_{\max\{\rho(x,x_Q)/t,a\}}^{\infty} \d\nu(r) \d\mu(x)\\
&=\int_{a}^\infty \int_M \ca_{B(x_Q,rt)} |u(x)-u_Q|^2\ \d\mu(x)\d\nu(r) \\
&\lesssim t^2 \int_{a}^\infty r^{\kappa+2} e^{\lambda rt} \int_M\ca_{B(x_Q,rt)} (|u(x)|^2+|\nabla u(x)|_{T^*_xM}^2)\ \d\mu(x)\d\nu(r) \\
&= t^2 \int_M (|u(x)|^2+|\nabla u(x)|_{T^*_xM}^2) \bigg(\int_{\max\{\rho(x,x_Q)/t,a\}}^\infty r^{\kappa+2} e^{\lambda rt} \d\nu(r) \bigg) \d\mu(x).\\
&= t^2 \int_M (|u(x)|^2+|\nabla u(x)|_{T^*_xM}^2) \bigg(\int_{\max\{\rho(x,x_Q)/t,a\}}^\infty \hspace{-2cm}r^{\kappa+2} e^{\lambda rt} (M r^{-M-1}+\tfrac{m}{a} r^{-M})e^{-\frac{m}{a} r}\ \d r\bigg) \d\mu(x).
\end{align*}
The term in brackets is bounded by
\[
e^{-(\frac{m}{a}-\lambda t) \rho(x,Q)/t} \int_{\rho(x,Q)/t}^\infty r^{-(M-(\kappa+2))}\ \d r
\lesssim e^{-(\frac{m}{a}-\lambda t) \rho(x,Q)/t} \left\langle\frac{t}{\rho(x,Q)}\right\rangle^{M-(\kappa+3)},
\]
which completes the proof.
\end{proof}

The first term in~\eqref{eq: red.est.3parts} is now estimated in a manner similar to that of Proposition~5.5 in \cite{AKMc}. The idea to replace the cube counting techniques used in \cite{AKMc} with the measure based approach below was suggested by Pascal Auscher.

\begin{Prop}\label{prop: 5.5} Let $C_\Theta>0$ be the constant from Proposition~\ref{Prop: ODestKato}. We have
\[\int_0^{\langle C_\Theta/4a^3\lambda \rangle} \|\Theta_t^BP_tu-\gamma_tA_tP_tu\|^2\ \frac{\d t}{t} \lesssim \|u\|^2\]
for all $u\in\Ran(\Pi)$.
\end{Prop}

\begin{proof}
Choose $M>4\kappa+3$ and let $t_0 = \langle C_\Theta/4a^3\lambda \rangle$. Let $t\in(0,t_0]$, $u\in\Ran(\Pi)$ and set $v=P_tu$.  The Cauchy--Schwartz inequality shows that
\begin{align*}
\|&\Theta_t^BP_tu-\gamma_tA_tP_tu\|^2
= \sum_{Q\in\Delta_t}\|\Theta_t^B\sum_{R\in\Delta_t}\ca_R(v-v_Q)\|_{L^2(Q)}^2
\\
&\leq \sum_{Q\in\Delta_t}\left(\sum_{R\in\Delta_t}\left(\frac{\mu(R)}{\mu(Q)} \frac{\mu(Q)}{\mu(R)}\right)^{\frac{1}{2}}\|\ca_Q\Theta_t^B\ca_R(v-v_Q)\|\right)^2
\\
&\leq \sup_{Q\in\Delta_t} \left(\sum_{R\in\Delta_t} \frac{\mu(R)}{\mu(Q)}\|\ca_Q\Theta_t^B\ca_R\|\right)\sum_{Q\in\Delta_t} \sum_{R\in\Delta_t}  \frac{\mu(Q)}{\mu(R)}\|\ca_Q\Theta_t^B\ca_R\|\|\ca_R(v-v_Q)\|^2.
\end{align*}
Then, since $C_\Theta>\lambda t$, Lemma~\ref{Lem: CubeSums} and the off-diagonal estimates from Proposition~\ref{Prop: ODestKato} show that the supremum term is uniformly bounded. Lemma~\ref{Lem: muQmuR} and \eqref{eq: xQ.x.QR} show that the remaining term is bounded by
\begin{align*}
&\sum_{Q\in\Delta_t} \sum_{R\in\Delta_t}  \left\langle\frac{t}{\rho(Q,R)}\right\rangle^{-\kappa} e^{a\lambda \rho(Q,R)} \left\langle\frac{t}{\rho(Q,R)}\right\rangle^M e^{-C_\Theta \frac{\rho(Q,R)}{t}} \|\ca_R(v-v_Q)\|^2\\
&\lesssim \sum_{Q\in\Delta_t} \sum_{R\in\Delta_t} \int_R \left\langle\frac{t}{\rho(Q,R)}\right\rangle^{M-\kappa} e^{-(C_\Theta-a\lambda t_0) \frac{\rho(Q,R)}{t}} |v(x)-v_Q|^2\ \d\mu(x)
\\
&\lesssim \sum_{Q\in\Delta_t} \int_M \left\langle\frac{t}{\rho(Q,x)}\right\rangle^{M-\kappa} e^{-(\frac{C_\Theta}{a}-\lambda t_0) \frac{\rho(Q,x)}{t}} |v(x)-v_Q|^2\ \d\mu(x).
\end{align*}
Using the notation from Section~\ref{Section: DiracTypeOperators} for functions in $L^2(M;\C^N)$, write $v=\sum_{\alpha=1}^N v_\alpha e_\alpha$. The weighted Poincar\'{e} inequality from Lemma~\ref{Lem: Poincare}, Lemma~\ref{Lem: muQmuR} and (H8) then show that the above estimate is bounded by
\begin{align*}
&t^2 \sum_{Q\in\Delta_t} \int_M \left\langle\frac{t}{\rho(Q,x)}\right\rangle^{M-(2\kappa+3)} e^{-(\frac{C_\Theta}{a^2}-(\frac{\lambda}{a}+\lambda)t_0) \frac{\rho(Q,x)}{t}} (|v(x)|^2+{\textstyle\sum_{\alpha}} |\nabla v_\alpha(x)|_{T^*_xM}^2)\ \d\mu(x)
\\
&\leq t^2 \sum_{R\in\Delta_t} (\|\ca_R v\|^2 + {\textstyle \sum_{\alpha}} \|\ca_R \nabla v_\alpha\|_{L^2(T^*M)}^2) \sum_{Q\in\Delta_t} \left\langle\frac{t}{\rho(Q,R)}\right\rangle^{M-(2\kappa+3)} e^{-(\frac{C_\Theta}{a^2}-2\lambda t_0)\frac{\rho(Q,R)}{t}}
\\
&\lesssim t^2 \|v\|_{W^{1,2}(M;\C^N)}^2 \sup_{R\in\Delta_t} \sum_{Q\in\Delta_t} \frac{\mu(Q)}{\mu(R)} \left\langle\frac{t}{\rho(Q,R)}\right\rangle^{M-(3\kappa+3)} e^{-(\frac{C_\Theta}{a^2}-3a\lambda t_0) \frac{\rho(Q,R)}{t}}
\\
&\lesssim t^2 \|v\|_{W^{1,2}(M;\C^N)}^2 \\
&\lesssim t^2 \|\Pi v\|^2,
\end{align*}
where the penultimate inequality is implied by Lemma~\ref{Lem: CubeSums} because $M-(3\kappa+3)>\kappa$ and $\frac{C_\Theta}{a^2}-3a\lambda t_0 > \lambda t$. Therefore, we have
\[
\|\Theta_t^BP_tu-\gamma_tA_tP_tu\|^2 \lesssim \|Q_tu\|^2
\]
for all $u\in\Ran(\Pi)$ and $t\in(0,t_0]$. The result then follows from the quadratic estimate for the unperturbed operator in \eqref{eq: unper.quad.est}.
\end{proof}

The following interpolation inequality is used to estimate the remaining terms in~\eqref{eq: red.est.3parts}. It is an extension of Lemma~6 in \cite{AKMc2}. The result relies on having a certain control of the volume of dyadic cubes near their boundary. This control is given by property (6) in Proposition~\ref{Prop: dyadic.cubes}.

\begin{Lem}\label{lem: interpolation.inequality}
Let $\Upsilon$ denote either $\Pi, \Gamma$ or $\Gamma^*$, then
\[\left|\dashint_Q \Upsilon u\right|^2 \lesssim \frac{1}{l(Q)^\eta} \left(\dashint_Q |u|^2\right)^{\frac{\eta}{2}} \left(\dashint_Q |\Upsilon u|^2\right)^{1-\frac{\eta}{2}} + \dashint_Q |u|^2\]
for all $u\in\Dom(\Upsilon)$, $Q\in\Delta_t$ and $t\in(0,1]$, where $\eta>0$ is from Proposition~\ref{Prop: dyadic.cubes}.
\end{Lem}
\begin{proof}
Let $s=\|\ca_Q u\|/ \|\ca_Q \Upsilon u\|$. If $s\geq a_0 l(Q)/2$, then the Cauchy--Schwartz inequality implies that
\begin{align*}
\left|\dashint_Q \Upsilon u\right|^2 
&\leq \dashint_Q |\Upsilon u|^2 \\
&=\frac{s^{-\eta}}{\mu(Q)} \left(\int_Q |u|^2\right)^{\frac{\eta}{2}} \left(\int_Q|\Upsilon u|^2\right)^{1-\frac{\eta}{2}} \\
&\lesssim \frac{1}{l(Q)^\eta} \left(\dashint_Q |u|^2\right)^{\frac{\eta}{2}} \left(\dashint_Q|\Upsilon u|^2\right)^{1-\frac{\eta}{2}}.
\end{align*}

Now suppose that $0<s\leq a_0l(Q)/2$. Let $Q_s=\{x\in Q : \rho(x,M\setminus Q) > s\} \subset Q$. It follows from Proposition~\ref{Prop: dyadic.cubes} that there exists $c>0$ such that
\[
\mu(M\setminus Q_s) \leq c(s/l(Q))^\eta \mu(Q)
\]
Choose $\eta:M\rightarrow[0,1]$ in $C^\infty_c(M)$ satisfying $\supp \eta \subseteq Q$ as well as
\[
\eta(x)=\begin{cases}
    1, &{\rm if}\quad x\in Q_s; \\
    0, &{\rm if}\quad x\in M\setminus Q
  \end{cases}
\]
and $\|\nabla \eta\|_{\infty} \lesssim 1/s$. The existence of such functions follows as in the proof of Proposition~\ref{Prop: ODestKato}. Using (H6)--(H7), we then obtain
\begin{align*}
\left|\int_Q \Upsilon u\right|
&= \left|\int_Q [\eta,\Upsilon] u  + \int_Q (1-\eta)\Upsilon u + \int_Q \Upsilon(\eta u)\right| \\
&\lesssim \|\nabla\eta\|_\infty \int_{\supp(\nabla\eta)}|u| + \int_{Q\cap\supp(1-\eta)}|\Upsilon u| + \mu(Q)^{\frac{1}{2}}\left(\int_Q |u|^2\right)^{\frac{1}{2}}\\
&\lesssim \mu(M\setminus Q_s)^{\frac{1}{2}} \big(\|\ca_Q u\|/s + \|\ca_Q\Upsilon u\|\big) + \mu(Q)^{\frac{1}{2}} \|\ca_Q u\|\\
&\lesssim (s/l(Q))^{\frac{\eta}{2}} \mu(Q)^{\frac{1}{2}} \|\ca_Q\Upsilon u\| + \mu(Q)^{\frac{1}{2}}\|\ca_Q u\|.
\end{align*}
This shows that
\begin{align*}
\left|\dashint_Q \Upsilon u\right|^2
&\lesssim \frac{1}{l(Q)^\eta} \frac{s^\eta}{\mu(Q)} \int_Q|\Upsilon u|^2 + \dashint_Q |u|^2 \\
&= \frac{1}{l(Q)^\eta} \left(\dashint_Q |u|^2\right)^{\eta/2} \left(\dashint_Q|\Upsilon u|^2\right)^{1-\eta/2} + \dashint_Q |u|^2,
\end{align*}
as required.
\end{proof}

The second term in~\eqref{eq: red.est.3parts} is now estimated by following the proof of Proposition~5 in \cite{AKMc2}.

\begin{Prop}\label{prop: 5.7}
We have
\[
\int_0^1 \|\gamma_tA_t(P_t-I)u\|^2\ \frac{\d t}{t} \lesssim \|u\|^2
\]
for all $u\in L^2(M; \C^N)$.
\end{Prop}
\begin{proof}
Lemma~\ref{lem: interpolation.inequality} and H\"{o}lder's inequality imply that
\begin{align*}
\|A_tQ_su\|^2 &= s^2\sum_{Q\in\Delta_t} \mu(Q) \left|\dashint_Q \Pi P_s u\right|^2 \\
&\lesssim s^2\sum_{Q\in\Delta_t} \frac{\mu(Q)}{l(Q)^\eta} \left(\dashint_Q |P_s u|^2\right)^{\frac{\eta}{2}} \left(\dashint_Q |\Pi P_s u|^2\right)^{1-\frac{\eta}{2}} + s^2\|P_s u\|^2\\
&\lesssim \left(\frac{s}{t}\right)^\eta \sum_{Q\in\Delta_t} \left(\int_Q |P_su|^2\right)^{\frac{\eta}{2}} \left(\int_Q |Q_su|^2\right)^{1-\frac{\eta}{2}} + s^2\|P_s u\|^2\\
&\leq \left(\frac{s}{t}\right)^\eta \|P_su\|^\eta \|Q_su\|^{2-\eta} + t^2\left(\frac{s}{t}\right)^2\|u\|^2\\
&\lesssim (\tfrac{s}{t})^\eta \|u\|^2
\end{align*}
for all $u \in L^2(M; \C^N)$ and $0<s<t\leq 1$. The result then follows by the arguments in the proof of Proposition~5 in \cite{AKMc2}.
\end{proof}

To estimate the third and final term in~\eqref{eq: red.est.3parts}, it follows from Theorem~\ref{thm: local.carleson.tentspaceequiv} that it suffices to show that there exists $t_0\in(0,1]$ such that
\begin{equation}\label{eq: Tb0}
\iint_{C(Q)} |\gamma_t(x)|^2\ \d\mu(x)\frac{\d t}{t} \lesssim \mu(Q)
\end{equation}
for all dyadic cubes $Q\in \bigcup_{t\in(0,t_0]}\Delta_t$.

Following \cite{AKMc}, we let $\sigma>0$ to be fixed later. Given $v\in \mathcal{L}(\C^N)$ with $|v|=1$, define the cone of aperture $\sigma$ by
\[
K_{v,\sigma}:=\{v'\in \mathcal{L}(\C^N)\setminus \{0\} : \left| \frac{v'}{|v'|}-v\right| \leq \sigma\}.
\]
Let $\mathcal{V}_\sigma$ be a finite set of $v\in\mathcal{L}(\C^N)$ with $|v|=1$ such that $\bigcup_{v\in\mathcal{V}_\sigma} K_{v,\sigma} = \mathcal{L}(\C^N)\setminus\{0\}$.
To prove \eqref{eq: Tb0}, it suffices to prove that there exists $t_0>0$ and $\sigma>0$ such that
\begin{equation}\label{eq: Tb1}
\iint_{\substack{(x,t) \in C(Q) \\ \gamma_t(x)\in K_{v,\sigma}}} |\gamma_t(x)|^2\ \d\mu(x)\frac{\d t}{t} \lesssim \mu(Q)
\end{equation}
for each $v\in\mathcal{V}_\sigma$ and for all $Q\in \bigcup_{t\in(0,t_0]}\Delta_t$. This in turn reduces to proving the following proposition.

\begin{Prop}\label{Prop: Carleson.sawtooth}
Let $t_0= \langle C_\Theta/4a^3\lambda\rangle$ , where $C_\Theta>0$ is the constant from Proposition~\ref{Prop: ODestKato}. There exist $\sigma,\tau,c>0$ such that for all $Q\in\bigcup_{t\in(0,t_0]}\Delta_t$ and $v\in\mathcal{L}(\C^N)$ with $|v|=1$, there exists a collection $\{Q_k\}_k\subseteq \Delta_{(0,1]}$ of disjoint subsets of $Q$ such that the set $E_{Q}:=Q\setminus\bigcup_k Q_k$ satisfies $\mu(E_{Q})>\tau\mu(Q)$ and the set  $E^*_{Q}:=C(Q)\setminus\bigcup_k C(Q_k)$ satisfies
\begin{equation*}
\iint_{\substack{(x,t) \in E^*_{Q} \\ \gamma_t(x)\in K_{v,\sigma}}} |\gamma_t(x)|^2\ \d\mu(x)\frac{\d t}{t} \leq c \mu(Q).
\end{equation*}
\end{Prop}

To see that Proposition~\ref{Prop: Carleson.sawtooth} implies \eqref{eq: Tb1}, write
\begin{align*}
\{(x,t) \in C(Q) : \gamma_t(x) \in K_{v,\sigma}\} &= E_{Q}^* \cup \left(\bigcup_{k_1} \{(x,t) \in C(Q_{k_1}) : \gamma_t(x) \in K_{v,\sigma}\} \right) \\
&= E_{Q}^* \cup E_{Q_{k_1}}^*\!\! \cup \left(\bigcup_{k_2} \{(x,t) \in C(Q_{k_2}) : \gamma_t(x) \in K_{v,\sigma}\} \right) \\
&= \bigcup_{j=0}^\infty \bigcup_{k_j=0}^\infty E_{{Q_{k_j}}}^*.
\end{align*}
Monotone convergence then implies that
\begin{align*}
\iint_{\substack{(x,t) \in C(Q) \\ \gamma_t(x)\in K_{v,\sigma}}} |\gamma_t(x)|^2\ \d\mu(x)\frac{\d t}{t}
&= \iint\sum_{j=0}^\infty \sum_{k_j=0}^\infty  \ca_{E_{{Q_{k_j}}}^*}(x,t) |\gamma_t(x)|^2\ \d\mu(x)\frac{\d t}{t} \\
&= \sum_{j=0}^\infty \sum_{k_j=0}^\infty  \iint_{E_{{Q_{k_j}}}^*}|\gamma_t(x)|^2\ \d\mu(x)\frac{\d t}{t} \\
&\lesssim \sum_{j=0}^\infty \sum_{k_j=0}^\infty \mu(Q_{k_j}) \\
&= \sum_{j=0}^\infty \mu(\bigcup_{k_j=0}^\infty Q_{k_j}) \\
&< \sum_{j=0}^\infty (1-\tau)^j \mu(Q) \\
&= \frac{1}{\tau} \mu(Q).
\end{align*}

The proof of Proposition~\ref{Prop: Carleson.sawtooth} is a matter of constructing suitable test functions and applying a stopping-time argument. The test functions are constructed as in~\cite{AKMc}, with some minor modifications. Fix $v\in \mathcal{L}(\C^N)$ with $|v|=1$ and choose $\hat{w},w\in\C^N$ such that $|\hat{w}|=|w|=1$ and $v^*(\hat{w})=w$. For each $Q\in\bigcup_{t\in(0,1]}\Delta_t$, let $B_Q$ denote a ball of radius $a_1 l(Q)$ such that $(a_0/a_1)B_Q \subseteq Q \subseteq B_Q$. Then let $\eta_Q:M\rightarrow [0,1]$ be a smooth function supported on $3B_Q$ and equal to 1 on $2B_Q$. Define $w_Q:=\eta_Q w$, and for each $\ep>0$, define the test function
\[
f^w_{Q,\ep} := w_Q-i\ep l(Q)\Gamma R^B_{\ep l(Q)}w_Q
= (I+i\ep l(Q)\Gamma_B^*)R^B_{\ep l(Q)}w_Q.
\]
These functions have the following properties. The proof is almost identical to that of Lemma~7 in \cite{AKMc2} but we include it for completeness.

\begin{Lem}\label{lem: 5.10}
There exists $c>0$ such that the following hold for all $Q\in\Delta_{(0,1]}$ and $\ep>0$:
\begin{enumerate}
\item $\displaystyle \|f^w_{Q,\ep}\| \leq c\mu(Q)^{\frac{1}{2}}$;
\item $\displaystyle \iint_{C(Q)} |\Theta^B_t f^w_{Q,\ep}|^2\ \d\mu(x)\frac{\d t}{t} \leq c\ep^{-2}\mu(Q)$;
\item $\displaystyle \left|\dashint_Q f_{Q,\ep}^w-w\right| < c\ep^{\frac{\eta}{2}}$,
\end{enumerate}
where $\eta>0$ is the constant Proposition~\ref{Prop: dyadic.cubes}.
\end{Lem}
\begin{proof}
1. Let $Q\in\bigcup_{t\in(0,1]} \Delta_t$. Using \eqref{eq: GGB}, Proposition~\ref{Prop: ODestKato} and \eqref{E}, we obtain
\[
\|f^w_{Q,\ep}\| \lesssim \|\eta_Q\|+\|i\ep l(Q)\Pi_BR^B_{\ep l(Q)}\eta_Q\|
\lesssim \|\eta_Q\| \leq \mu(2B_Q)^{1/2}
\lesssim \mu(Q)^{1/2},
\]
where the constant in the last inequality is uniform for all $Q\in\bigcup_{t\in(0,1]} \Delta_t$.

2. Next, by the nilpotency of $\Gamma_B^*$ and $[\Gamma_B^*,P^B_t]=0$ on $\Dom(\Gamma_B^*)$, we have
\begin{align*}
\Theta^B_t f_{Q,\ep}^w &= tP^B_t\Gamma_B^*(I+i\ep l(Q)\Gamma_B^*)R^B_{i\ep l(Q)}w_Q = tP^B_t\Gamma_B^*R^B_{i\ep l(Q)}w_Q.
\end{align*}
Therefore, using \eqref{eq: GGB}, Proposition~\ref{Prop: ODestKato} and \eqref{E} again, we obtain
\begin{align*}
\iint_{C(Q)} |\Theta^B_t f^w_{Q,\ep}|^2\ \d\mu(x)\frac{\d t}{t}
&\leq \int_0^{l(Q)} \|tP^B_t\Gamma_B^*R^B_{i\ep l(Q)}w_Q\|^2\ \frac{\d t}{t} \\
&\lesssim \int_0^{l(Q)} \frac{t}{(\ep l(Q))^2} \|i\ep l(Q)\Pi_BR^B_{i\ep l(Q)}\eta_Q\|^2\ \d t \\
&\lesssim \frac{1}{\ep^2} \mu(Q).
\end{align*}

3. Finally, since $\eta_Q=1$ on $Q$, by Lemma~\ref{lem: interpolation.inequality} with $\Upsilon=\Gamma$ and $u=R^B_{\ep l(Q)}w_Q$, and using \eqref{eq: GGB}, Proposition~\ref{Prop: ODestKato} and \eqref{E} again, we obtain
\begin{align*}
&\left|\dashint_Q f_{Q,\ep}^w-w\right|
= \ep l(Q)\left|\dashint_Q \Gamma R^B_{\ep l(Q)}w_Q\right| \\
&\lesssim \ep l(Q)^{1-\frac{\eta}{2}} \left(\dashint_Q | R^B_{\ep l(Q)} w_Q|^2\right)^{\frac{\eta}{4}} \left(\dashint_Q |\Gamma R^B_{\ep l(Q)} w_Q|^2\right)^{\frac{1}{2}-\frac{\eta}{4}} + \ep l(Q)\left(\dashint_Q |R^B_{\ep l(Q)}w_Q|^2\right)^{\frac{1}{2}}\\
&\lesssim \mu(Q)^{-\frac{1}{2}} \|R_{\ep l(Q)}^Bw_Q\|^{\frac{\eta}{2}} \left(\int_Q |i\ep l(Q)\Pi_B R^B_{\ep l(Q)} w_Q|^2\right)^{\frac{1}{2}-\frac{\eta}{4}} + \ep \mu(Q)^{-\frac{1}{2}}\|R^B_{\ep l(Q)}w_Q\|\\
&\lesssim \ep^{\frac{\eta}{2}} \mu(Q)^{-\frac{1}{2}} \|R_{\ep l(Q)}^Bw_Q\|^{\frac{\eta}{2}}
\|(I-R_{\ep l(Q)}^B)w_Q\|^{1-\frac{\eta}{2}}  + \ep \mu(Q)^{-\frac{1}{2}}\|\eta_Q\|\\
&\lesssim \ep^{\frac{\eta}{2}}\mu(Q)^{-\frac{1}{2}} \|\eta_Q\| \\
&\lesssim \ep^{\frac{\eta}{2}},
\end{align*}
as required.
\end{proof}

We now fix $\ep=(\tfrac{1}{2c})^{2/\eta}$ and the test functions $f_Q^w:=f_Q^{w,\ep}$, where $c$ is the constant from Lemma~\ref{lem: 5.10}. The preceding result then implies that
\[
\text{Re}\left(w,\dashint_Q f_Q^w\right) \geq \frac{1}{2}.
\]
The stopping-time argument from Lemma~5.11 in \cite{AKMc} can then be applied to obtain the following result. The properties of the dyadic cube structure in Proposition~\ref{Prop: dyadic.cubes} suffice for this purpose.
\begin{Lem}\label{lem: 5.11}
Let $t_0=\langle C_\Theta/4a^3\lambda \rangle$. There exist $\alpha,\beta>0$ such that for all dyadic cubes $Q\in\bigcup_{t\in(0,t_0]}\Delta_t$ there exists a collection $\{Q_k\}_k\subseteq \Delta_{(0,1]}$ of disjoint subsets of $Q$ such that the set $E_Q:=Q\setminus\bigcup_k Q_k$ satisfies $\mu(E_Q)>\beta\mu(Q)$ and the set $E^*_Q:=C(Q)\setminus\bigcup_k C(Q_k)$ has the following property:
\[
\text{Re}\left(w,\dashint_{Q'} f_Q^w\right) \geq \alpha
\quad\text{and}\quad
\dashint_{Q'} |f^w_Q| \leq \frac{1}{\alpha}
\]
for all $Q'\in\Delta_{(0,1]}$ that are contained in $Q$ and satisfy $C(Q')\cap E^*_Q\neq\emptyset$.
\end{Lem}

We can now prove Proposition~\ref{Prop: Carleson.sawtooth} by following closely the ideas at the end of Section~5 in \cite{AKMc}.

\begin{proof}[Proof of Proposition~\ref{Prop: Carleson.sawtooth}]
Choose $\sigma\in(0,\alpha^2)$ and let $\tau=\beta$, where $\alpha,\beta>0$ are the constants from Lemma~\ref{lem: 5.11}.

Let $Q\in\bigcup_{t\in(0,t_0]}\Delta_t$ and $v\in\mathcal{L}(\C^N)$ with $|v|=1$. Let $\{Q_k\}_k\subseteq \Delta_{(0,1]}$ denote the collection of disjoint subsets of $Q$ given by Lemma~\ref{lem: 5.11} and suppose that $(x,t)\in E^*_Q$. This implies that $(x,t)\in C(Q)$ and that $t\leq l(Q)\leq t_0/\delta$. Now let $Q'$ be the unique dyadic cube in $\Delta_t$ that contains $x$. Then, since $l(Q')\geq t$, we must have $(x,t)\in C(Q')$ and so $C(Q')\cap E^*_{Q}\neq \emptyset$. Lemma~\ref{lem: 5.11} and the Cauchy--Schwartz inequality then imply that
\[
|v(A_tf^w_Q(x))| \geq \re (\hat{w},v(A_tf^w_Q(x))) = \re \left(w,\dashint_{Q'} f^w_Q(x)\right) \geq \alpha
\]
and that
\[
|A_tf^w_Q(x)| = \left|\dashint_{Q'} f^w_Q(x)\right| \leq \frac{1}{\alpha}.
\]
The choice of $\sigma$ then implies that
\[
\left|\frac{\gamma_t(x)}{|\gamma_t(x)|} A_tf^w_Q(x)\right|
\geq |v(A_tf^w_Q(x))| - \left|\frac{\gamma_t(x)}{|\gamma_t(x)|}-v\right| \left|A_tf^w_Q(x)\right|
\geq \alpha - \frac{\sigma}{\alpha} \gtrsim 1.
\]
Therefore, we have
\begin{align*}
\iint_{\substack{(x,t) \in E^*_{Q} \\ \gamma_t(x)\in K_{v,\sigma}}} |\gamma_t(x)|^2\ \d&\mu(x)\frac{\d t}{t}
\lesssim \iint_{C(Q)} |\gamma_t(x)A_tf_Q^w(x)|^2\ \d\mu(x)\frac{\d t}{t} \\
&\lesssim \iint_{C(Q)} |\Theta^B_tf^w_Q-\gamma_tA_tf_Q^w|^2\ \d\mu\frac{\d t}{t} +\iint_{C(Q)} |\Theta^B_tf^w_Q|^2\ \d\mu\frac{\d t}{t}.
\end{align*}
Lemma~\ref{lem: 5.10} shows that the last term above is bounded by $c(2c)^{4/\eta}\mu(Q)$. It remains to show that
\[
\iint_{C(Q)} |\Theta^B_tf^w_Q-\gamma_tA_tf_Q^w|^2\ \d\mu\frac{\d t}{t} \lesssim \mu(Q).
\]

Now let $v=i\ep l(Q)\Gamma R^B_{\ep l(Q)}w_Q$ and write
\begin{equation}\label{eq: zzz}
\Theta^B_tf^w_Q-\gamma_tA_tf_Q^w = -(\Theta^B_t-\gamma_tA_t)v + (\Theta^B_t-\gamma_tA_t)w_Q.
\end{equation}
Then, since $v\in\Ran(\Gamma)$, by (i) in Proposition~4.8 of \cite{AKMc}, Proposition~\ref{prop: 5.5} and Proposition~\ref{prop: 5.7}, we have
\begin{align*}
\iint_{C(Q)} |(\Theta^B_t-\gamma_tA_t)v|^2\ \d\mu\frac{\d t}{t}
& \lesssim \int_0^{t_0} \|\Theta^B_t(I-P_t)v\|^2 \frac{\d t}{t}\\
&\quad+\int_0^{t_0} \|(\Theta^B_tP_t-\gamma_tA_tP_t)v\|^2 \frac{\d t}{t} \\
&\quad+\int_0^{t_0} \|\gamma_tA_t(P_t-I)v\|^2 \frac{\d t}{t} \\
&\lesssim \mu(Q).
\end{align*}
To handle the remaining term in \eqref{eq: zzz}, recall that $(a_0/a_1)B_Q\subseteq Q\subseteq B_Q$ and that $\eta_Q=1$ on $2B_Q$. This implies that if $x\in Q$ and $t\in(0,l(Q)]$, then
\[
(\Theta^B_t-\gamma_tA_t)w_Q(x) = \Theta^B_t((\eta_Q-1)w)(x).
\]
Now choose $M>\kappa/2$ and consider the characteristic functions $\ca_j(B_Q)$ defined by
\[
\ca_j(B_Q)=
\begin{cases}
\ca_{2B_Q}&{\rm if}\quad j=0;\\
\ca_{2^{j+1}B_Q\setminus 2^{j}B_Q}&{\rm if}\quad j=1,2,\ldots.
\end{cases}
\]
Then, since $\eta_Q-1=0$ on $2B_Q$, the off-diagonal estimates from Proposition~\ref{Prop: ODestKato} and $\eqref{E}$ imply that
\begin{align*}
\|\Theta^B_t(\eta_Q-1)w\|_{L^2(Q)}^2 &\leq \sum_{j=1}^\infty \|\ca_{B_Q}\Theta^B_t\ca_j(B_Q)\|^2 \|\ca_j(B_Q)(\eta_Q-1)\|^2 \\
&\lesssim \sum_{j=1}^\infty \left(\frac{t}{(2^j-1)a_1l(Q)}\right)^{2M}e^{-2C_\Theta(2^j-1)a_1l(Q)/t}\mu(2^{j+1}B_Q)\\
&\lesssim \frac{t}{l(Q)} \mu(B_Q) \sum_{j=1}^\infty 2^{-j(2M-\kappa)}e^{-(C_\Theta-\lambda t_0)2^{j+1}a_1l(Q)/t}\\
&\leq \frac{t}{l(Q)} \mu(Q)
\end{align*}
for all $t\in(0,l(Q)]$. This shows that
\[
\iint_{C(Q)} |(\Theta^B_t-\gamma_tA_t)w_Q|^2\ \d\mu\frac{\d t}{t} \lesssim \mu(Q),
 \]
so the proof is complete.
\end{proof}

As shown previously, Proposition~\ref{Prop: Carleson.sawtooth} implies \eqref{eq: Tb1}, which in turn implies \eqref{eq: Tb0} and allows us to estimate the final term in \eqref{eq: red.est.3parts}. In summary, as a consequence of Propositions~\ref{prop: 5.5}, \ref{prop: 5.7} and \ref{Prop: Carleson.sawtooth}, we have proved the local quadratic estimate
\[
\int_0^{\langle C_\Theta/4a^3\lambda\rangle} \|\Theta_t^B P_tu\|^2\ \frac{\d t}{t} \lesssim \|u\|^2
\]
for all $u\in\Ran(\Gamma)$. The hypothesis (H1)--(H8) are invariant upon replacing $\{\Gamma,B_1,B_2\}$ with $\{\Gamma^*,B_2,B_1\}$, $\{\Gamma^*,B_2^*,B_1^*\}$ and $\{\Gamma,B_1^*,B_2^*\}$. This completes the proof of the main result, since Proposition~\ref{Prop: t0reduction} then allows us to conclude that the quadratic estimate in Theorem~\ref{Thm: mainquadraticestimate} holds.

\section*{Acknowledgements}
This work was conducted at the Centre for Mathematics and its Applications at the Australian National University, and at the Department of Mathematics at the University of Missouri. I was supported by these institutions, as well as by the Australian Government through an Australian Postgraduate Award.

I would like to thank Andreas Axelsson for suggesting that I consider the Kato square root problem on a submanifold of Euclidean space. I am especially grateful to Alan McIntosh for suggesting the approach taken in this paper, and for his patience and dedication as a supervisor that made it all possible. I would also like to thank Ben Andrews for his kindness and expertise in submanifold geometry that brought the paper together.

A version of this paper formed a chapter in my PhD thesis, and I would like to thank my examiners for suggestions that improved the work enormously. It is also a pleasure to thank Pascal Auscher, Lashi Bandara and Chema Martell whose helpful conversations and suggestions have contributed significantly to this work.

\bibliographystyle{amsplain}

\begin{thebibliography}{10}

\bibitem{AbreschMeyer1997}
Uwe Abresch and Wolfgang~T. Meyer, \emph{Injectivity radius estimates and
  sphere theorems}, {C}omparison {G}eometry ({B}erkeley, {CA}, 1993--94), Math.
  Sci. Res. Inst. Publ., vol.~30, Cambridge Univ. Press, Cambridge, 1997,
  pp.~1--47.

\bibitem{Aubin1976}
Thierry Aubin, \emph{Espaces de {S}obolev sur les vari\'et\'es riemanniennes},
  Bull. Sci. Math. (2) \textbf{100} (1976), no.~2, 149--173.

\bibitem{AHLMT}
Pascal Auscher, Steve Hofmann, Michael Lacey, Alan McIntosh, and Ph.
  Tchamitchian, \emph{The solution of the {K}ato square root problem for second
  order elliptic operators on {$\mathbb{R}\sp n$}}, Ann. of Math. (2)
  \textbf{156} (2002), no.~2, 633--654.

\bibitem{AMcN1997}
Pascal Auscher, Alan McIntosh, and Andrea Nahmod, \emph{The square root problem
  of {K}ato in one dimension, and first order elliptic systems}, Indiana Univ.
  Math. J. \textbf{46} (1997), no.~3, 659--695.

\bibitem{AKMc2}
Andreas Axelsson, Stephen Keith, and Alan McIntosh, \emph{The {K}ato square
  root problem for mixed boundary value problems}, J. London Math. Soc. (2)
  \textbf{74} (2006), no.~1, 113--130.

\bibitem{AKMc}
\bysame, \emph{Quadratic estimates and functional calculi of perturbed {D}irac
  operators}, Invent. Math. \textbf{163} (2006), no.~3, 455--497.

\bibitem{AzagraFerreraLopezRangel2007}
D.~Azagra, J.~Ferrera, F.~L{\'o}pez-Mesas, and Y.~Rangel, \emph{Smooth
  approximation of {L}ipschitz functions on {R}iemannian manifolds}, J. Math.
  Anal. Appl. \textbf{326} (2007), no.~2, 1370--1378.

\bibitem{BrMiSh2002}
M.~Braverman, O.~Milatovich, and M.~Shubin, \emph{Essential selfadjointness of
  {S}chr\"odinger-type operators on manifolds}, Uspekhi Mat. Nauk \textbf{57}
  (2002), no.~4(346), 3--58.

\bibitem{Borsuk1947}
Karol Borsuk, \emph{Sur la courbure totale des courbes ferm\'ees}, Ann. Soc.
  Polon. Math. \textbf{20} (1947), 251--265 (1948).

\bibitem{Buser1982}
Peter Buser, \emph{A note on the isoperimetric constant}, Ann. Sci. \'Ecole
  Norm. Sup. (4) \textbf{15} (1982), no.~2, 213--230.

\bibitem{CMcM}
Andrea Carbonaro, Alan McIntosh, and Andrew Morris, \emph{Local {H}ardy spaces
  of differential forms on {R}iemannian manifolds}, arXiv:1004.0018 (2010).

\bibitem{Chavel2006}
Isaac Chavel, \emph{Riemannian {G}eometry: {A} {M}odern {I}ntroduction}, second
  ed., Cambridge Studies in Advanced Mathematics, vol.~98.

\bibitem{CheegerEbin1975}
Jeff Cheeger and David~G. Ebin, \emph{{C}omparison {T}heorems in {R}iemannian
  {G}eometry}, North-Holland Publishing Co., Amsterdam, 1975, North-Holland
  Mathematical Library, Vol. 9.

\bibitem{ChowKnopf2004}
Bennett Chow and Dan Knopf, \emph{The {R}icci {F}low: {A}n {I}ntroduction},
  Mathematical Surveys and Monographs, vol. 110, American Mathematical Society,
  Providence, RI, 2004.

\bibitem{Christ}
Michael Christ, \emph{A {$T(b)$} theorem with remarks on analytic capacity and
  the {C}auchy integral}, Colloq. Math. \textbf{60/61} (1990), no.~2, 601--628.

\bibitem{CMcM1982}
R.~R. Coifman, A.~McIntosh, and Y.~Meyer, \emph{L'int\'egrale de {C}auchy
  d\'efinit un op\'erateur born\'e sur {$L^{2}$} pour les courbes
  lipschitziennes}, Ann. of Math. (2) \textbf{116} (1982), no.~2, 361--387.

\bibitem{Fenchel1929}
Werner Fenchel, \emph{\"{U}ber kr\"{u}mmung und windung geschlossener
  raumkurven}, Math. Ann. \textbf{101} (1929), no.~1, 238--252.

\bibitem{Folland1999}
Gerald~B. Folland, \emph{{R}eal {A}nalysis: {M}odern {T}echniques and their
  {A}pplications}, second ed., Pure and Applied Mathematics, John Wiley \& Sons
  Inc., New York, 1999.

\bibitem{Haase2006}
Markus Haase, \emph{{T}he {F}unctional {C}alculus for {S}ectorial {O}perators},
  Operator Theory: Advances and Applications, vol. 169, Birkh\"auser Verlag,
  Basel, 2006.

\bibitem{Hebey1999}
Emmanuel Hebey, \emph{{N}onlinear {A}nalysis on {M}anifolds: {S}obolev {S}paces
  and {I}nequalities}, Courant Lecture Notes in Mathematics, vol.~5, New York
  University Courant Institute of Mathematical Sciences, New York, 1999.

\bibitem{HigsonRoe2000}
Nigel Higson and John Roe, \emph{{A}nalytic {$K$}-{H}omology}, Oxford
  Mathematical Monographs, Oxford University Press, Oxford, 2000, Oxford
  Science Publications.

\bibitem{HMc}
Steve Hofmann and Alan McIntosh, \emph{The solution of the {K}ato problem in
  two dimensions}, Proceedings of the 6th {I}nternational {C}onference on
  {H}armonic {A}nalysis and {P}artial {D}ifferential {E}quations ({E}l
  {E}scorial, 2000), no. Vol. Extra, 2002, pp.~143--160.

\bibitem{Kato1961}
Tosio Kato, \emph{Fractional powers of dissipative operators}, J. Math. Soc.
  Japan \textbf{13} (1961), 246--274.

\bibitem{Kato}
\bysame, \emph{{P}erturbation {T}heory for {L}inear {O}perators}, Classics in
  Mathematics, Springer-Verlag, Berlin, 1995, Reprint of the 1980 edition.

\bibitem{Lee}
John~M. Lee, \emph{{R}iemannian {M}anifolds: An {I}ntroduction to {C}urvature},
  Graduate Texts in Mathematics, vol. 176, Springer-Verlag, New York, 1997.

\bibitem{McIntosh1986}
Alan McIntosh, \emph{Operators which have an {$H\sb \infty$} functional
  calculus}, Miniconference on operator theory and partial differential
  equations (North Ryde, 1986), Proc. Centre Math. Anal. Austral. Nat. Univ.,
  vol.~14, Austral. Nat. Univ., Canberra, 1986, pp.~210--231.

\bibitem{McIntosh1990}
\bysame, \emph{The square root problem for elliptic operators: {A} survey},
  Functional-analytic methods for partial differential equations ({T}okyo,
  1989), Lecture Notes in Math., vol. 1450, Springer, Berlin, 1990,
  pp.~122--140.

\bibitem{Morris2010(1)}
Andrew~J. Morris, \emph{Local quadratic estimates and holomorphic functional
  calculi}, AMSI--ANU Workshop on Spectral Theory and Harmonic Analysis (ANU,
  July 2009), Proc. Centre Math. Appl. Austral. Nat. Univ., vol.~44, Austral.
  Nat. Univ., Canberra, 2010, pp.~211--231.

\bibitem{Sakai1996}
Takashi Sakai, \emph{{R}iemannian {G}eometry}, Translations of Mathematical
  Monographs, vol. 149, American Mathematical Society, Providence, RI, 1996,
  Translated from the 1992 Japanese original by the author.

\bibitem{Saloff2002}
Laurent Saloff-Coste, \emph{Aspects of {S}obolev-{T}ype {I}nequalities}, London
  Mathematical Society Lecture Note Series, vol. 289, Cambridge University
  Press, Cambridge, 2002.

\end{thebibliography}

\end{document}